\documentclass[leqno]{article}

\usepackage{amsmath,amsthm,amsfonts,amssymb,amsfonts,graphics}

\usepackage[colorlinks=true]{hyperref}

\theoremstyle{plain}
\newtheorem{theorem}{Theorem}[section]
\newtheorem{definition}[theorem]{Definition}

\newtheorem{lemma}[theorem]{Lemma}
\newtheorem{corollary}[theorem]{Corollary}
\newtheorem{proposition}[theorem]{Proposition}

\theoremstyle{remark}

\def\R{{\mathbb R}}
\def\N{{\mathbb N}}

\def\virgp{\raise 2pt\hbox{,}}

\def\({\left(}
\def\){\right)}
\def\<{\left\langle}
\def\>{\right\rangle}
\def\le{\leqslant}
\def\ge{\geqslant}
\def\les{\lesssim}
\def\dn{\partial_\nu}

\def\Tend#1#2{\mathop{\longrightarrow}\limits_{#1\rightarrow#2}}

\def\d{{\partial}}

\def\eps{\varepsilon}

\DeclareMathOperator{\Lmeq}{{\mathcal L}(m_{\rm eq})}
\DeclareMathOperator{\Ltmeq}{{\mathcal L}(t,m_{\rm eq}(t))}

\numberwithin{equation}{section}

\begin{document}


\title{Hysteresis for ferromagnetism: asymptotics 
of some 2-scale Landau-Lifshitz model}   
\author{Eric Dumas\footnote{Universit\'e Grenoble 1, Institut Fourier, 
100 rue des maths-BP~74, 38402 Saint Martin d'H\`eres cedex, France; 
edumas@ujf-grenoble.fr.}, 
St\'ephane Labb\'e\footnote{Universit\'e Grenoble 1, Laboratoire Jean 
Kuntzmann, 51 rue des maths, 38402 Saint Martin d'H\`eres cedex, 
France; Stephane.Labbe@imag.fr. 
Support by The Nano-Science Foundation, Grenoble, Project HM-MAG is acknowledged.}}

\maketitle

\noindent
{\bf Abstract.} 
We study a 2-scale version of the Landau-Lifshitz 
system of ferromagnetism, introduced by Starynkevitch to modelize 
hysteresis: the response of the magnetization is fast compared to 
a slowly varying applied magnetic field. Taking the exchange term 
into account, in space dimension 3, we prove that, under some 
natural stability assumption on the equilibria of the system, 
the strong solutions follow the dynamics of these equilibria. 
We also give explicit examples of relevant equilibria 
and exterior magnetic fields, when the ferromagnetic medium 
occupies some ellipsoidal domain. 

\tableofcontents

\section{Introduction}
\label{sec:intro}

Hysteresis is a widely studied, yet not completely understood 
phenomenon. It has played a role from the very beginning of the 
works on magnetism. 
Lord Rayleigh \cite{Rayleigh87} proposed a model 
for ferromagnetic hysteresis in 1887, while the most achieved  
micromagnetism theory goes back to Landau and Lifshitz, in 1935 
(see \cite{LL69}). 

In \cite{Visintin94}, Visintin gives many historical references, 
underlines the links between several forms of hysteresis  
(in particular, from plasticity, and from ferromagnetism),  
and how it is related to phase transitions. He performs a 
mathematical study of the so-called hysteresis operators, 
including the most famous one, due to Preisach. 

Recently in \cite{CEF09}, Carbou, Effendiev and Fabrie have proved 
the existence of strong solutions to a model of ferromagnetic 
hysteresis due to Effendiev. 

In this paper, we rather investigate properties of a two-scale  
model introduced by Starynkevitch in \cite{StarynPhD}. 
This model describes the dynamics obtained when some exterior 
magnetic field is applied to the ferromagnetic material 
under consideration, while the response of the magnetization 
occurs on a much shorter time scale (say, denoted by $\eps>0$). 
Mathematically, such models, associated to ordinary differential 
equations, had been studied in the nonstandard analysis framework, 
leading to ``canard cycles'' (see \cite{DR89}). Considering a 
Landau-Lifshitz model in 0 space dimension (thus, an ODE), 
Starynkevitch studies the possible equilibria of the system, 
and the asymptotic behavior of the solutions (as the above mentioned 
parameter $\eps$ goes to zero) when the exterior magnetic field 
slowly varies. 

Our aim is to extend Starynkevitch's approach to the  
Landau-Lifshitz model in space dimension three, taking exchange 
term into account. This means, giving the asymptotic description 
of solutions to the slow-fast corresponding system of partial 
differential equations. 
Here, we prove such a result away from the bifurcation points of 
hysteresis loops. More precisely, assuming that the system 
(described by its magnetization) possesses at each time $t$ 
some stable equilibrium $m_{\rm eq}(t)$, 
and is submitted to some slowly varying exterior magnetic field, 
we show that the magnetization follows the dynamics of $m_{\rm eq}$. We also give explicit examples (for ellipsoidal domains) 
of relevant equilibria and exterior magnetic fields. 

\section{Statement of the results}
\label{sec:result}

The initial and boundary value problem associated to the 
2 scale Landau-Lifshitz equation considered reads:
\begin{equation}
\label{eq:ll}
\left\{  
\begin{array}{ll}
& \eps \d_t m^\eps = m^\eps \wedge h_T^\eps - \alpha m^\eps \wedge 
(m^\eps \wedge h_T^\eps) , \, \text{ for } t\ge0, \, x\in\Omega,  \\
& \dn m^\eps_{|_{\d\Omega}} = 0 , \\
& m^\eps_{|_{t=0}} = m_0 . 
\end{array} 
\right.
\end{equation}
The unknown is the magnetization $m^\eps$, function of the 
time variable $t\ge0$ and of the space variable $x\in\Omega$, with 
values in the sphere $S^2\subset\R^3$. The domain $\Omega$ occupied 
by the ferromagnet is a subset of $\R^3$. Furthermore, 
$h_T^\eps=h_T(t,m^\eps(t))$, where the total magnetic field $h_T$ 
is defined by
\begin{equation}
\label{eq:hT}
h_T(t,m) = \overline{\Delta m} + h_{\rm d}(m) + h_{\rm ext}(t).
\end{equation}
Here, the first term $\Delta m$ is the ``exchange term'', which 
tends to impose a constant magnetization (domains where magnetization 
is constant are called ``Weiss domains''), and  $\overline{\Delta m}$ 
denotes the extension of $\Delta m$ by $0$ out of $\Omega$. 
The second term, yielding spatial variations of the 
magnetization, is the ``demagnetizing field'' $h_{\rm d}(m)$, 
which results from a quasi-stationary approximation 
of Maxwell's equations; it is defined 
(at least, for $m \in L^2(\Omega,\R^3)$, 
as an element of $L^2(\R^3,\R^3)$) by 
\begin{equation*}
\mathop{\rm curl} h_{\rm d}(m) = 0 \quad \text{and} \quad 
\mathop{\rm div} \left( h_{\rm d}(m) + \overline m \right) = 0 
\quad \text{in } \R^3 .
\end{equation*}
Classical properties of the mapping $m \mapsto h_{\rm d}(m)$ 
are recalled in Section~\ref{sec:fctalfacts}. The third term,  
$h_{\rm ext}$, denotes some given exterior field, which is assumed to 
depend on time (and possibly on space). The positive constant $\alpha$ 
is some damping coefficient, which appears in the model when passing 
from a microscopic to a macroscopic description. The small parameter 
$\eps>0$ expresses the fact that, while the exterior field $h_{\rm ext}$ 
depends on $t$, and has time variations at scale 1, the magnetization 
$m^\eps$ essentially depends on $t/\eps$, and thus has variations at 
the much more rapid scale $\eps$. 

Throughout this paper, for any $s\in\N$, we denote by $H^s(\Omega)$ 
the usual Sobolev space of functions with values in some vector space 
$\R^N$, whereas $H^s(\Omega,S^2)$ is the Sobolev space of functions 
with values in the sphere $S^2$ (which is not a vector space),
$$
H^s(\Omega,S^2) = \{ m \in H^s(\Omega) \mid 
|m| \equiv 1 \text{ almost everywhere} \}. 
$$
Finally, for $s\ge2$, $H^s_N(\Omega)$ denotes the subspace 
of functions in $H^s(\Omega)$ with homogeneous von Neumann 
boundary condition,
$$
H^s_N(\Omega) = \{ m \in H^s(\Omega) \mid  
\dn m_{|_{\d\Omega}} = 0 \}, 
$$
and 
$$
H^s_N(\Omega,S^2) = H^s(\Omega,S^2) \cap H^s_N(\Omega) .
$$
All these spaces (even if not vector spaces) inherit the (metric) 
topology given by the usual norm on $H^s(\Omega)$.

We prove the following 
\begin{theorem} \label{th:asympt}
Let $\Omega$ be an open and bounded subset of $\R^3$, 
with smooth boundary. 
Let $T>0$, and $h_{\rm ext} \in C^1([0,T],C^\infty(\R^3))$, 
bounded with bounded derivatives. 
Assume that there exist $m_{\rm eq} \in C^1([0,T],H^2_N(\Omega,S^2))$ 
and $m_0 \in H^2_N(\Omega,S^2)$ such that \\
(i) for all $t_0\in[0,T]$, $m_{\rm eq}(t_0)$ is an equilibrium 
for 
\begin{equation} \label{eq:llt0} 
\d_t m = m \wedge h_T(t_0,m) - \alpha m \wedge 
\Big( m \wedge h_T(t_0,m) \Big) 
\end{equation}
(see \eqref{eq:caractmeq}); \\
(ii) the solution $n_0$ to the initial and boundary value problem 
\begin{equation} \label{eq:n0} \left\{
\begin{array}{ll}
& \d_t n_0 = n_0 \wedge h_T(0,n_0) - \alpha n_0 \wedge 
\Big( n_0 \wedge h_T(0,n_0) \Big) , \\
& \dn {n_0}_{|_{\d\Omega}} = 0 , \\
& {n_0}_{|_{t=0}} = m_0 ,
\end{array} \right.
\end{equation}
is global ($n_0 \in C([0,\infty),H^2_N(\Omega,S^2))$), 
with $\nabla\Delta n_0 \in L^2((0,\infty)\times\Omega)$,  
and $n_0(t)$ converges in $H^2(\Omega)$, as $t$ goes to $\infty$, 
towards $m_{\rm eq}(0)$; \\ 
(iii) the linearized operator $\Lmeq$ given by 
\eqref{eq:defLeq} has the following dissipation property:
\begin{equation} \label{eq:hypdissip}
\begin{split}
& \text{there exist } C_{\rm lin}>0 \text{ and } \eta>0 
\text{ such that}, \\ 
& \text{for all } \delta \in C([0,T],H^\infty(\Omega)) 
\text{ with } |m_{\rm eq}+\delta| \equiv 1 \text{ and } 
\dn \delta_{|_{\d\Omega}} = \dn \Delta \delta_{|_{\d\Omega}} = 0, \\ 
& \sup_{t\in\lbrack0,T\rbrack} \|\delta(t)\|_{H^2(\Omega)} \le \eta 
\text{ implies:} \\
& \forall t \in [0,T], \quad 
\Big( \Ltmeq\delta(t) \mid \delta(t) \Big)_{H^2(\Omega)} 
\le -C_{\rm lin} \|\delta(t)\|_{H^2(\Omega)}^2.
\end{split}
\end{equation}
Then, there is $\eps_0>0$ such that, 
for all $\eps \in (0,\eps_0)$, 
the solution $m^\eps$ to \eqref{eq:ll} exists up to time $T$ ($m^\eps \in C([0,T],H^2_N(\Omega,S^2))$), 
and converges in $L^2((0,T),H^2(\Omega)) \cap C([t,T],H^2(\Omega))$ towards $m_{\rm eq}$ 
as $\eps$ goes to zero, for all $t\in(0,T)$.
\end{theorem}

To prove Theorem~\ref{th:asympt}, we first show that $m^\eps$ 
converges to $m_{\rm eq}(0)$ within an initial layer of size 
$t_\eps = C \eps \ln(1/\eps)$. This is achieved \emph{via} 
classical energy estimates (in $H^2$), carefully controlling 
the dependence upon $\eps$ --more technically speaking, the 
quasilinear and elliptic degenerate system of PDE's in \eqref{eq:ll} 
is first converted into a perturbation of some linear, strongly 
elliptic system, yielding the usual smooting properties, and 
a Galerkine approximation is used. 
In a second step, we prove that $m^\eps$ converges towards 
$m_{\rm eq}$ on the whole time interval $[t_\eps,T]$. 
This amount to proving of long-time existence and return to 
equilibrium result for small initial data. 
Toward this end, we use again energy estimates, 
together with the stability assumption \eqref{eq:hypdissip}.

Figure \ref{dynamics} illustrates this
corresponding asymptotic behaviour.
\begin{figure}[ht]
\includegraphics{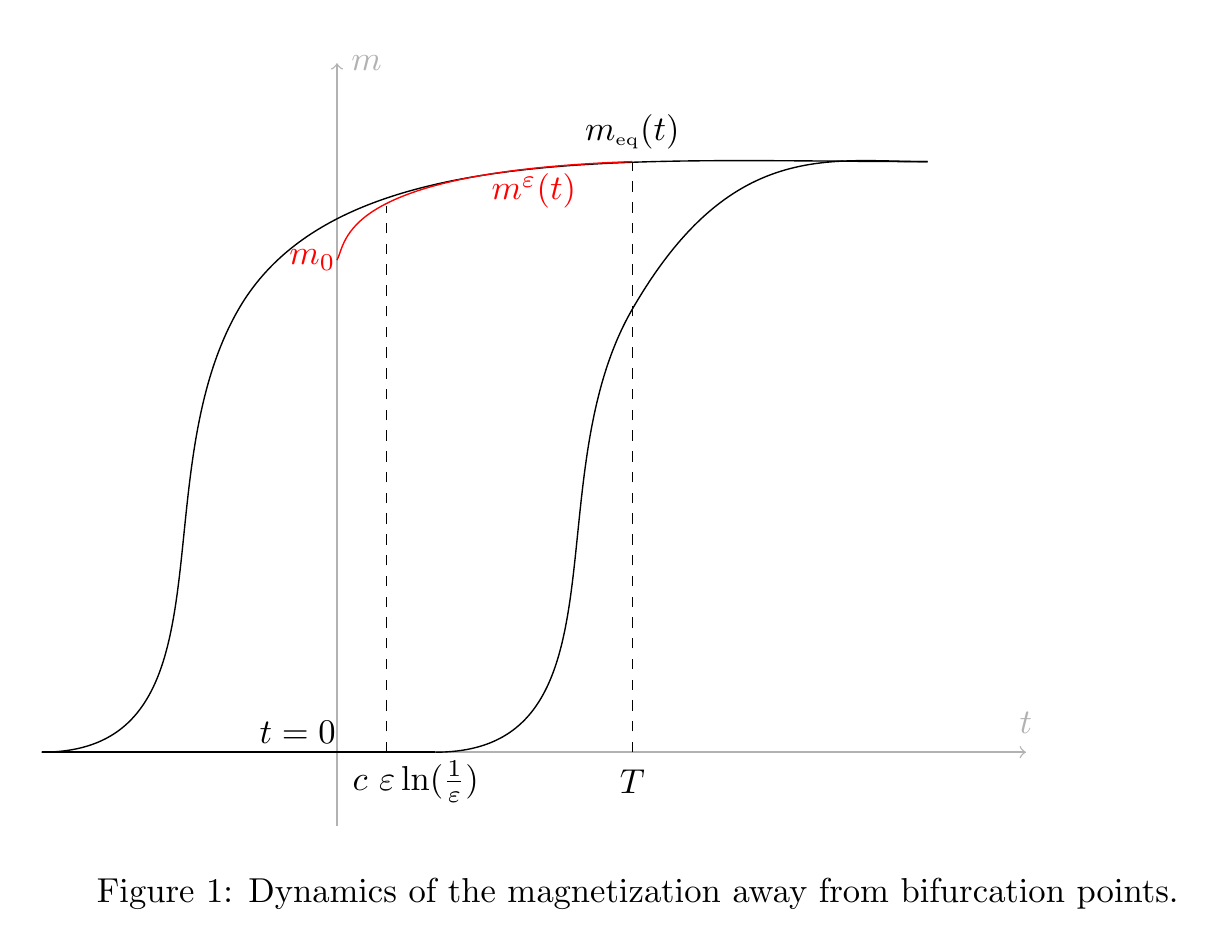}
\caption{Dynamics of the magnetization away from bifurcation points.}
\label{dynamics}
\end{figure}

The above assumptions on the equilibrium $m_{\rm eq}$ are 
discussed in Section~\ref{sec:equilibria} below. 
In particular, Assumption (ii) in Theorem~\ref{th:asympt} 
may be understood as a choice of `prepared' data $m_0$ 
allowing to deal with the initial layer $(0,c\eps\ln(1/\eps))$. 
The dissipation property \eqref{eq:hypdissip} expresses, 
for all $t_0 \in \lbrack 0,T \rbrack$, the stability 
of the linearization around $m_{\rm eq}(t_0)$ of \eqref{eq:ll}, 
with $\eps=1$ and with $h_{\rm ext}$ replaced with 
$h_{\rm ext}(t_0)$, independent of time. This is a strong 
assumption, which ensures global existence of the solutions 
to the corresponding Landau-Lifshitz equation, for initial data 
close to $m_{\rm eq}(t_0)$:
\begin{proposition} \label{prop:globalexist}
Let $\Omega$ be an open and bounded subset of $\R^3$, 
with smooth boundary. Consider an exterior 
magnetic field $h_{\rm ext} \in C^\infty(\R^3)$ 
(independent of time) bounded with bounded derivatives. 
Assume that there exists $m_{\rm eq} \in H^2_N(\Omega,S^2)$ 
(independent of time) satisfying the equilibrium condition 
\begin{equation} \label{hyp:equil}
m_{\rm eq} \wedge 
( \Delta m_{\rm eq} + h_{\rm d}(m_{\rm eq}) + h_{\rm ext} ) = 0 
\quad \text{on } \Omega ,
\end{equation}
as well as the stability condition 
\begin{equation} \label{hyp:stabil}
\begin{split}
& \text{there exist } C_{\rm lin}>0 \text{ and } \eta>0 
\text{ such that}, \\ 
& \text{for all } \delta \in H^\infty(\Omega) 
\text{ with } |m_{\rm eq}+\delta| \equiv 1 \text{ and } 
\dn \delta_{|_{\d\Omega}} = \dn \Delta \delta_{|_{\d\Omega}} = 0, \\ 
& \|\delta(t)\|_{H^2(\Omega)} \le \eta 
\text{ implies:} \\
& \Big( \mathcal L(0,m_{\rm eq})\delta \mid 
\delta \Big)_{H^2(\Omega)} 
\le -C_{\rm lin} \|\delta\|_{H^2(\Omega)}^2,
\end{split}
\end{equation}
for the linearized operator $\mathcal L(0,m_{\rm eq})$ 
given by \eqref{eq:defLeq} (with $m_{\rm eq}(0)$ and 
$h_{\rm ext}(0)$ replaced with $m_{\rm eq}$ and $h_{\rm ext}$, 
respectively). 
 
Then, there exists $\eta_0>0$ such that, for all 
$m_0 \in H^2_N(\Omega,S^2)$ satisfying 
$$
\| m_0 - m_{\rm eq} \|_{H^2(\Omega)} \le \eta_0 ,
$$
the solution $n$ to the initial and boundary value problem 
\begin{equation} \label{eq:n0bis} 
\left\{
\begin{array}{ll}
& \d_t n = n \wedge h_T(0,n) - \alpha n \wedge 
\Big( n \wedge h_T(0,n) \Big) , \\
& \dn {n}_{|_{\d\Omega}} = 0 , \\
& {n}_{|_{t=0}} = m_0 ,
\end{array} 
\right.
\end{equation}
is global ($n \in C([0,\infty),H^2_N(\Omega,S^2))$), with 
$\nabla\Delta(n-m_{\rm eq})\in L^2((0,\infty)\times\Omega)$,  
and $n(t)$ converges in $H^2(\Omega)$, as $t$ goes to $\infty$, 
towards $m_{\rm eq}$.
\end{proposition}
In the case of $m_{\rm eq}(0)$ constant over $\Omega$, 
Proposition~\ref{prop:globalexist} expresses that 
in Theorem~\ref{th:asympt}, assumptions (i) and (iii) imply 
assumption (ii), so that we get:
\begin{corollary} \label{cor:asymptbis}
Let $\Omega$ be an open and bounded subset of $\R^3$, 
with smooth boundary. 
Let $T>0$, and $h_{\rm ext} \in C^1([0,T],C^\infty(\R^3))$, 
bounded with bounded derivatives. 
Assume that there exist $m_{\rm eq} \in C^1([0,T],H^2_N(\Omega,S^2))$ satisfying assumptions (i) and 
(iii) from Theorem~\ref{th:asympt}. Assume furthermore that 
$m_{\rm eq}(0)$ is constant over $\Omega$. 

Then, there exist $\eta_0,\eps_0>0$ such that, for all 
$m_0 \in H^2_N(\Omega,S^2)$ such that 
$$
\| m_0 - m_{\rm eq}(0) \|_{H^2(\Omega)} \le \eta_0 ,
$$
and for all $\eps \in (0,\eps_0)$, 
the solution $m^\eps$ to \eqref{eq:ll} exists up to time $T$ ($m^\eps \in C([0,T],H^2_N(\Omega,S^2))$), 
and converges in $L^2((0,T),H^2(\Omega)) \cap C([t,T],H^2(\Omega))$ towards $m_{\rm eq}$ 
as $\eps$ goes to zero, for all $t\in(0,T)$.
\end{corollary}
In Lemma~\ref{lem:dissip} below, we give examples 
(in ellipsoidal domains) of equilibria $m_{\rm eq}$ satisfying 
the assumptions of Corollary~\ref{cor:asymptbis}.


\section{Preliminaries}
\label{sec:prelim}

\subsection{Some functional analysis}
\label{sec:fctalfacts}

In this section, we recall some functional analysis results 
useful in the sequel.
The first of them deals with the continuity properties of 
the demagnetizating field operator $h_{\rm d}$, immediately 
deduced from the Fourier representation 
$\displaystyle \widehat{h_{\rm d}(u)}(\xi) = 
- \left(\xi\cdot\hat{\bar u}(\xi)\right)\frac{\xi}{|\xi|^2}$: 
\begin{lemma}[$h_d$ properties] 
Let $\Omega$ be an open subset of $\mathbb{R}^ 3$. 
For all $s$ in $\mathbb{N}$ and $u$ in $H^s(\Omega)$, one has
\begin{eqnarray*}
\Vert h_{\rm d}(u)\Vert_{H^s(\mathbb{R}^3)}\le \Vert u\Vert_{H^s(\Omega)}.
\end{eqnarray*}
Furthermore, for all $v$ in $L^2(\Omega)$ we have
\begin{eqnarray*}
(h_{\rm d}(u) \mid v)_{L^2(\Omega)} = 
-(u \mid h_{\rm d}(v))_{L^2(\Omega)}.
\end{eqnarray*}
\end{lemma}

In addition to the usual Sobolev embeddings, 
we recall the following estimate, which results 
from the coercivity of the operator $A = 1-\Delta$, 
with domain $D(A) = \{ m \in H^2(\Omega) \mid 
\dn m_{|_{\d\Omega}}=0 \}$ (see for example \cite{MR902801})
\begin{lemma} \label{lem:sobolev}
Let $\Omega$ be a smooth bounded open set in $\mathbb{R}^ 3$. 
There exists a constant $C>0$ such that for all $u$ in 
$H_N^2(\Omega)$ one has
\begin{equation*}
\Vert u \Vert_{L^\infty(\Omega)} \le C\left(	\Vert u\Vert_{L^2(\Omega)}^2 + \Vert\Delta u\Vert_{L^2(\Omega)}^2\right)^{\frac{1}{2}}.
\end{equation*}
\end{lemma}
In the sequel, we will need the following definition.
\begin{definition} \label{def:Pk}
Let $\Omega$ be a smooth bounded open set in $\mathbb{R}^ 3$. 
For $k\in\N^\star$, let $P_k$ be the $L^2(\Omega)$-orthogonal 
projection onto $V_k$, the vector space spanned by the first 
$k$ eigenfunctions of $A=1-\Delta$, with domain 
$D(A) = \{ m \in H^2(\Omega) \mid \dn m_{|_{\d\Omega}}=0 \}$. 
\end{definition}
The family of operators $(P_k)_{k\in\mathbb{N}}$ satisfies 
useful properties:
\begin{lemma} \label{lem:propPk}
The following properties are true.
\begin{itemize}
\item [(i)] $\forall k \in \N^\star$, $\forall u\in D(A)$, 
$\Delta P_k u = P_k \Delta u$,
\item [(ii)] $\forall k \in \N^\star$, $\forall s\in\mathbb{N}$, 
$\forall u\in H^s(\Omega)$, $P_k u \in H^s(\Omega)$ \\
(and $P_k u \in H^s_N(\Omega)$ when $s\ge2$),
\item[(iii)] $\forall s\in\mathbb{N}$, 
$\displaystyle \lim_{k\rightarrow\infty}
\Vert (1-P_k) u\Vert_{H^s(\Omega)}=0$ 
for all $u \in H^s(\Omega)$ when $s=0,1$, 
and for all $u \in H^s_N(\Omega)$ when $s\ge2$.
\end{itemize}
\end{lemma}
\begin{proof}
(i) For all $u$ in $D(A)$, $k$ in $\mathbb{N}^\star$, one has
$$
P_k \Delta u = \sum_{j=1}^k (\Delta u \mid \psi_j)_{L^2} \psi_j,
$$
with $(\psi_i)_{i\in\mathbb{N}}$ the $L^2$-orthonormal basis 
of the eigenvectors of $\Delta$ associated to the eigenvalues 
$(\lambda_i)_{i\in\mathbb{N}}$. 
Then, using the vanishing Neumann boundary conditions, 
$$
(\Delta u \mid \psi_j)_{L^2} = - \int_\Omega \nabla u \cdot \nabla \psi_j   = (u \mid \Delta \psi_j)_{L^2},
$$
so that
$$
P_k \Delta u 
= \sum_{j=1}^k \lambda_j (u \mid \psi_j)_{L^2} \psi_j  
= \Delta P_k u.
$$
Point (ii) follows from the regularity properties 
of the family $(\psi_i)_{i\in\mathbb{N}}$. \\
Point (iii) is a consequence of the fact that 
$u \mapsto \left( 
\sum_{j=1}^\infty (1+\lambda_j^s)|(u \mid \psi_j)_{L^2}|^2
\right)^{1/2}$ provides a norm equivalent to the usual one 
on $H^s(\Omega)$. 
\end{proof}

\subsection{About equilibria} \label{sec:equilibria}

\paragraph{Global solutions and equilibria.} 
In \cite[Th. 4.3]{AB09}, in the case of ellipsoidal domains 
$\Omega\subset\R^3$ and under a smallness assumption 
(on $\| h_{\rm ext} \|_{L^\infty}$ and $\| \Delta m_0 \|_{L^2}$), 
Alouges and Beauchard construct global smooth solutions to 
\eqref{eq:ll}. 
Furthermore, these solutions satisfy 
$$\forall T>0, \quad 
\| \Delta m(T) \|_{L^2(\Omega)}^2 + C \int_0^T 
\| \nabla\Delta m \|_{L^2(\Omega)}^2
\le \| \Delta m_0 \|_{L^2(\Omega)}, $$
so that $\nabla\Delta m$ belongs to $L^2((0,\infty)\times\Omega)$. 
This is a part of our assumptions on the equilibrium 
$m_{\rm eq}$, 
when requiring the existence of the global solution $n_0$. 
Saying that $m_{\rm eq}(t_0)$ is an equilibrium 
for \eqref{eq:llt0} means 
\begin{equation}
\label{eq:caractmeq}
\left\{  
\begin{array}{ll}
& m_{\rm eq}(t_0) \wedge h_T(t_0,m_{\rm eq}(t_0)) = 0 ,  \\
& {\dn m_{\rm eq}(t_0)}_{|_{\d\Omega}} = 0 ,
\end{array} \right.
\end{equation} 
and requiring $H^2$ convergence of $n_0(t)$ 
towards $m_{\rm eq}(0)$ as $t$ goes to $\infty$ implies 
that $m_{\rm eq}(0)$ is an equilibrium for \eqref{eq:llt0} 
with $t_0=0$.
\paragraph{Energy minimization.} 
It is worth noting that energy decay occurs along the evolution 
of $n_0(t)$, so that one may hope at least $H^1$ convergence 
of $n_0(t)$ towards some local minimum of the energy, 
as $t$ goes to $\infty$. 
To the Landau-Lifshitz system \eqref{eq:ll} is associated 
the energy 
$$\mathcal{E}(t,m) = \frac{1}{2} \int_\Omega |\nabla m|^2 
- \frac{1}{2} \int_\Omega m \cdot h_{\rm d}(m)  
- \int_\Omega m \cdot h_{\rm ext}(t) ,$$
and when $m$ is solution to \eqref{eq:ll}, we have
$$\frac{\rm d}{\rm dt}\mathcal{E}(t,m(t)) = 
-\frac{\alpha}{\varepsilon} 
\| m(t) \wedge h_T(t,m(t)) \|_{L^2(\Omega)}^2 
- \int_\Omega m(t) \cdot \d_t h_{\rm ext}(t) .$$
Since the exterior magnetic field does not depend on time 
during the evolution of $n_0$, we get 
$$\frac{\rm d}{\rm dt}\mathcal{E}(t,n_0(t)) = 
- \alpha \| n_0(t) \wedge h_T(t,n_0(t)) \|_{L^2(\Omega)}^2.$$

In the case of ellipsoidal domains, 
special configurations are available. 
See \cite{Osborn45}, and references therein: 
there exists a real $3\times3$ definite 
positive diagonalizable matrix $D$ giving the demagnetizing field 
resulting from any magnetization constant constant over $\Omega$: 
$$
\forall v\in\R^3, \quad h_{\rm d}(v)_{|_\Omega} \equiv -Dv .
$$
Hence, if $u\in S^2$ is an eigenvector of $D$ associated to the 
eigenvalue $d>0$, and if the exterior magnetic field is 
$h_{\rm ext} = \lambda u$ for some $\lambda>0$ 
(or $h_{\rm ext}(x) = \lambda \chi(x) u$ for some 
$\chi \in C^\infty_c(\R^3,\lbrack 0,1 \rbrack)$ to get a spatially 
localized field), then the system possesses two explicit equilibria 
$m_{\rm eq}^+$ and $m_{\rm eq}^-$:
\begin{equation} \label{eq:defmeqpm}
m_{\rm eq}^\pm = \pm u.
\end{equation}
One easily computes the energy associated to perturbations of these 
equilibria: for all $\delta\in H^2_N(\Omega,\R^3)$ such that 
$|m_{\rm eq}^\pm+\delta|=1$ a.e.,
$$
\mathcal{E}(m_{\rm eq}^\pm+\delta) - \mathcal{E}(m_{\rm eq}^\pm) = 
\frac{1}{2} \int_\Omega |\nabla\delta|^2 
- \frac{1}{2} \int_\Omega \delta \cdot h_{\rm d}(\delta) 
+ \frac{1}{2} (\pm\lambda-d) \int_\Omega |\delta|^2.  
$$
The first two terms are non-negative, so that for $\lambda$ 
large enough ($\lambda>d$), $m_{\rm eq}^+$ is a global minimum 
of $\mathcal{E}$; but for $\lambda$ small, it may fail to be 
even a local minimum. Concerning $m_{\rm eq}^-$, for all 
$\lambda>0$, if $d$ is the largest eigenvalue of $D$, 
and $\delta$ is constant in space, then the difference of energies  
above is less than $\displaystyle -\frac{\lambda}{2} 
\int_\Omega |\delta|^2$, thus negative, whereas for $\delta$ 
with large variations, the gradient term dominates, 
and the energy difference becomes positive. 
Hence, $m_{\rm eq}^-$ is always a saddle point for $\mathcal{E}$.

\paragraph{The dissipation property~\eqref{eq:hypdissip}.}  
We have the following lemma, the proof of which is postponed 
to Section~\ref{sec:proofHypdissip}:
\begin{lemma} \label{lem:dissip}
For $\lambda>0$ large enough, the equilibrium 
$m_{\rm eq}^+$ from \eqref{eq:defmeqpm} satisfies 
the dissipation property~\eqref{eq:hypdissip} 
(for some constant $C_{\rm lin}$ depending on $\lambda$).
\end{lemma} 
For $m_{\rm eq}^-$, it is shown 
in Section~\ref{sec:proofHypdissip} that for $\lambda$ large, 
we have on the contrary: 
\begin{lemma} \label{lem:nondissip}
For $\lambda>0$ large enough, there exist $C=C(\alpha,\lambda)>0$ 
and $\eta=\eta(\alpha,\lambda)>0$ such that, for all 
$\delta \in C([0,T],H^\infty(\Omega)) 
\text{ with } |m_{\rm eq}+\delta| \equiv 1 \text{ and } 
\dn \delta_{|_{\d\Omega}} = \dn \Delta \delta_{|_{\d\Omega}} = 0$,  
when $\|\delta\|_{H^2(\Omega)} \le \eta$, we have:
$$
\forall t \in [0,T], \quad 
\Big( {\mathcal L}(t,m_{\rm eq}^-(t))\delta(t) \mid 
\delta(t) \Big)_{H^2(\Omega)} 
\ge C \|\delta\|_{H^2(\Omega)}^2.
$$ 
\end{lemma}

\section{Proof of Theorem~\ref{th:asympt}}
\label{sec:proofTh}

First, consider the solution  
$n_0$ to the Cauchy problem 
\begin{equation*} \left\{
\begin{array}{ll}
& \d_t n_0 = n_0 \wedge h_T(0,n_0) - \alpha n_0 \wedge \Big( n_0 \wedge h_T(0,n_0) \Big) , \\
& \dn {n_0}_{|_{\d\Omega}} = 0 , \\
& {n_0}_{|_{t=0}} = m_0 ,
\end{array} \right.
\end{equation*}
and define $n^\eps$ by
$$\forall t\ge0, \quad n^\eps(t)=n_0(t/\eps).$$
Then, $n^\eps \in C_b([0,\infty),H^2(\Omega))$ (with $\nabla\Delta n^\eps \in 
L^2((0,\infty)\times\Omega)$, and we know that 
\begin{equation} \label{eq:n0cv}
t_\eps/\eps \Tend{\eps}{0} \infty 
\quad \Longrightarrow \quad 
n^\eps(t_\eps) \Tend{\eps}{0} m_{\rm eq}(0) \text{ in } H^2(\Omega) . 
\end{equation}

Next, as in \cite{CF01}, we observe that for smooth functions $m$ 
with constant modulus (w.r.t. $x$), one has 
$m \cdot \Delta m = - 2 |\nabla m|^2$, so that smooth solutions to \eqref{eq:ll} 
equivalently satisfy
\begin{equation}
\label{eq:llparab}
\left\{  
\begin{array}{ll}
& \eps \d_t m^\eps - \alpha \Delta m^\eps = \mathcal{F}(t,m^\eps) , \\
& \dn m^\eps_{|_{\d\Omega}} = 0 , \\
& m^\eps_{|_{t=0}} = m_0 , 
\end{array} \right.
\end{equation}
where
\begin{equation} \label{eq:defF}
\mathcal{F}(t,m) = m \wedge h_T(t,m) + \alpha |\nabla m|^2 m
- \alpha m \wedge \Big( m \wedge (h_{\rm d}(m) + h_{\rm ext}(t)) \Big). 
\end{equation}
Furthermore, smooth ($L^\infty_t H^2_x$) solutions to 
\eqref{eq:llparab} issued from $m_0$ with constant modulus, 
equal to one, are shown to keep the same modulus for all time, 
(due to uniqueness of the solution $a=|m|^2$ to 
$\eps\d_t a = \alpha\Delta a + 2\alpha |\nabla u|^2 (a-1)$, 
$\dn a_{|_{\d\Omega}} = 0$, $a_{|_{t=0}}=1$). We thus solve 
\eqref{eq:llparab} in the Banach space 
$C([0,T],H^2_N(\Omega))$, and deduce from this conservation 
that the solution actually belongs to the space  $C([0,T],H^2_N(\Omega,S^2))$.

It is worth noting that \eqref{eq:ll} is an initial and boundary 
value problem for some quasilinear and parabolic degenerate 
operator, which is seen in \eqref{eq:llparab} as a perturbation 
of a linear and strongly parabolic one. 

Standard energy estimates ensure local-in-time existence and 
uniqueness of solutions continuous in time, with values in 
$H^2(\Omega))$ (with an existence time depending on $\eps$): 
see for example \cite{AB09} or \cite{CF01}.  
By the usual continuation argument, we simply 
need to bound the $H^2$ norm of $m^\eps$ to ensure existence 
up to time $T$. Actually, we shall prove convergence (as $\eps$ 
goes to zero) at the same time, \emph{via} energy estimates. 

We first show that, after some time $t_\eps$ 
of the form $t_\eps = C \eps \ln(1/\eps)$, $m^\eps$ and $n^\eps$ 
are close: $\sup_{[0,t_\eps]} \| m^\eps-n^\eps \|_{H^2(\Omega)}$ 
goes to zero with $\eps$; thus, for $\eps$ small enough, $m^\eps(t_\eps)$ is as close (in $H^2(\Omega)$) to $m_{\rm eq}(0)$ 
as desired. We then use the stability property of $m_{\rm eq}(t)$ 
to show that $m^\eps(t)$ stays close to it, for $t \in [t_\eps,T]$.

\subsection{First step: the initial layer \texorpdfstring{$[0,t_\eps]$}{[0,teps]}}

\subsubsection{Galerkine scheme} \label{sec:galerkine} 
For $k\in\N^\star$, let $P_k$ be the $L^2(\Omega)$-
orthogonal projection onto $V_k$, 
the vector space spanned by the first $k$ eigenfunctions 
of $A=1-\Delta$, with domain 
$D(A) = \{ m \in H^2(\Omega) \mid \dn m_{|_{\d\Omega}}=0 \}$, 
as in Definition~\ref{def:Pk}. 
Define a Galerkine approximation of \eqref{eq:llparab} by:
\begin{equation}
\label{eq:mkeps}
\left\{  
\begin{array}{ll}
& \eps \d_t m_k^\eps - \alpha \Delta m_k^\eps = 
P_k \mathcal{F}(t,m_k^\eps) ,  \\
& {m_k^\eps}_{|_{t=0}} = P_k m_0 .
\end{array} \right.
\end{equation}
The projection $n_k^\eps = P_k n^\eps$ also satisfies
\begin{equation*}
\begin{split}
\eps \d_t n_k^\eps - \alpha \Delta n_k^\eps = 
& P_k \mathcal{F}(0,n_k^\eps) + \alpha [P_k,\Delta] n^\eps 
+ P_k \Big( \mathcal{F}(0,n^\eps) - \mathcal{F}(0,n_k^\eps) \Big) \\
= & P_k \mathcal{F}(0,n_k^\eps) + P_k [P_k,\mathcal{F}(0,\cdot)](n^\eps) ,
\end{split}
\end{equation*}
since for $u \in D(A)$, $P_k \Delta u = \Delta P_k u$, 
according to Lemma~\ref{lem:propPk}.

Now, perform energy estimates (in $L^2$) for $\varphi_k^\eps = m_k^\eps - 
n_k^\eps$, solution to 
\begin{equation}
\label{eq:phikeps}
\left\{
\begin{array}{ll}
& \eps \d_t \varphi_k^\eps - \alpha \Delta \varphi_k^\eps = 
P_k \Big( \mathcal{F}(t,m_k^\eps) - \mathcal{F}(0,n_k^\eps) \Big) 
- P_k [P_k,\mathcal{F}(0,\cdot)](n^\eps) , \\
& {\varphi_k^\eps}_{|_{t=0}} = 0 .
\end{array}
\right.
\end{equation}

\subsubsection{\texorpdfstring{$L^2$}{L2} estimates} 

Take the scalar product (in $L^2(\Omega)$) of $\varphi_k^\eps$ with 
the first equation in \eqref{eq:phikeps} to get
\begin{equation*}
\frac{\eps}{2} \frac{\rm d}{\rm dt} 
\left( \| \varphi_k^\eps \|_{L^2(\Omega)}^2 \right) 
+ \alpha \| \nabla \varphi_k^\eps \|_{L^2(\Omega)}^2 = I_1 + I_2 + I_3 + I_4 , 
\end{equation*}
with
\begin{equation*}
\begin{split}
& I_1 = \Big( \varphi_k^\eps \mid m_k^\eps \wedge h_T(t,m_k^\eps) 
- n_k^\eps \wedge h_T(0,n_k^\eps) \Big)_{L^2(\Omega)} , \\
& I_2 = \alpha \Big( \varphi_k^\eps \mid  
|\nabla m_k^\eps|^2 m_k^\eps - |\nabla n_k^\eps|^2 n_k^\eps \Big)_{L^2(\Omega)} , \\
& I_3 = - \alpha \Big( \varphi_k^\eps \mid 
m_k^\eps \wedge \Big( m_k^\eps \wedge 
(h_{\rm d}(m_k^\eps) + h_{\rm ext}(t)) \Big) \\ 
& \hspace{3cm} - n_k^\eps \wedge \Big( n_k^\eps \wedge 
(h_{\rm d}(n_k^\eps) + h_{\rm ext}(0)) \Big) 
\Big)_{L^2(\Omega)} , \\
& I_4 = \Big( \varphi_k^\eps \mid 
[P_k,\mathcal{F}(0,\cdot)](n^\eps) \Big)_{L^2(\Omega)} .
\end{split}
\end{equation*}

\paragraph{\emph{Estimating $I_1$.}} 
Decompose $m_k^\eps = n_k^\eps + \varphi_k^\eps$. For all $\varphi,h\in\R^3$, 
$\varphi\cdot(\varphi\wedge h)=0$, so that
\begin{equation*}
\begin{split}
I_1 = 
& \Big( \varphi_k^\eps \mid n_k^\eps \wedge \Big( \Delta (n_k^\eps + \varphi_k^\eps) 
+ h_{\rm d}(n_k^\eps + \varphi_k^\eps) + h_{\rm ext}(t) \Big) \\ 
& \qquad\qquad\qquad\qquad -n_k^\eps \wedge \Big( \Delta n_k^\eps  + h_{\rm d}(n_k^\eps) 
+ h_{\rm ext}(0) \Big) \Big)_{L^2(\Omega)} \\
= & \Big( \varphi_k^\eps \mid n_k^\eps \wedge (\Delta \varphi_k^\eps 
+ h_{\rm d}(\varphi_k^\eps)) \Big)_{L^2(\Omega)} 
+ \Big( \varphi_k^\eps \mid n_k^\eps \wedge (h_{\rm ext}(t)-h_{\rm ext}(0)) \Big)_{L^2(\Omega)} .
\end{split}
\end{equation*}
Using the continuity of $h_{\rm d}$ on $L^2(\Omega)$, we get, for some constant $C$ depending on 
$\| \d_t h_{\rm ext} \|_{L^\infty_{t,x}}$ and $\| n_0 \|_{L^\infty((0,\infty),L^2(\Omega))}$: 
\begin{equation} \label{eq:estimI1}
I_1 \le C \| \varphi_k^\eps \|_{L^2(\Omega)} 
\Big( \| \varphi_k^\eps \|_{H^2(\Omega)} + t \Big).
\end{equation}

\paragraph{\emph{Estimating $I_2$.}}  
Write 
\begin{equation*}
\begin{split}
|\nabla m_k^\eps|^2 m_k^\eps - |\nabla n_k^\eps|^2 n_k^\eps
& = (|\nabla m_k^\eps|^2-|\nabla n_k^\eps|^2)m_k^\eps + |\nabla n_k^\eps|^2\varphi_k^\eps \\
& = (\nabla (2n_k^\eps +\varphi_k^\eps) \cdot \nabla \varphi_k^\eps) (n_k^\eps +\varphi_k^\eps)  
+ |\nabla n_k^\eps|^2\varphi_k^\eps. 
\end{split}
\end{equation*}
Then, use Sobolev's embeddings, such as
\begin{equation*}
\begin{split}
\Big( \varphi_k^\eps \mid 
( \nabla n_k^\eps \cdot \nabla \varphi_k^\eps) n_k^\eps \Big)_{L^2(\Omega)} 
& \le \| \varphi_k^\eps \|_{L^\infty(\Omega)} \| \nabla n_k^\eps \|_{L^2(\Omega)} 
\| \nabla\varphi_k^\eps \|_{L^4(\Omega)} \| n_k^\eps \|_{L^4(\Omega)} \\
& \les \| \varphi_k^\eps \|_{H^2(\Omega)} \| n_k^\eps \|_{H^1(\Omega)}
\| \nabla\varphi_k^\eps \|_{H^1(\Omega)} \| n_k^\eps \|_{H^1(\Omega)} ,
\end{split}
\end{equation*}
and 
\begin{equation*}
\Big( \varphi_k^\eps \mid |\nabla n_k^\eps|^2\varphi_k^\eps \Big)_{L^2(\Omega)} 
\le \| \varphi_k^\eps \|_{L^\infty(\Omega)}^2 \| \nabla n_k^\eps \|_{L^2(\Omega)}^2 
\les \| \varphi_k^\eps \|_{H^2(\Omega)}^2 \| n_k^\eps \|_{H^1(\Omega)}^2 ,
\end{equation*}
to get the estimate
\begin{equation} \label{eq:estimI2}
I_2 \le C \| \varphi_k^\eps \|_{H^2(\Omega)} \Big( \| \varphi_k^\eps \|_{H^2(\Omega)} 
+ \| \varphi_k^\eps \|_{H^2(\Omega)}^3 \Big) ,
\end{equation}
for some constant $C$ depending on $\| n_0 \|_{L^\infty((0,\infty),H^1(\Omega))}$.

\paragraph{\emph{Estimating $I_3$.}}  
As for $I_1$, cancellations allow to write
\begin{equation*}
\begin{split}
I_3 = 
& - \alpha \Big( \varphi_k^\eps \mid n_k^\eps \wedge 
\Big( n_k^\eps \wedge h_{\rm d} (\varphi_k^\eps) + \varphi_k^\eps \wedge h_{\rm d} (n_k^\eps) \Big)
+ n_k^\eps \wedge \Big( \varphi_k^\eps \wedge h_{\rm ext}(t) \Big) \Big)_{L^2(\Omega)} \\
& \quad - \alpha \Big( \varphi_k^\eps \mid n_k^\eps \wedge  
\Big( n_k^\eps \wedge (h_{\rm ext}(t)-h_{\rm ext}(0)) \Big) \Big)_{L^2(\Omega)} .
\end{split}
\end{equation*}
Boundedness of $h_{\rm d}$ on $L^p$ for finite $p$ provides the bounds
\begin{equation*}
\| n_k^\eps \wedge ( n_k^\eps \wedge h_{\rm d} (\varphi_k^\eps) ) \|_{L^2(\Omega)} 
\le \| n_k^\eps \|_{L^6(\Omega)}^2 \| \varphi_k^\eps \|_{L^6(\Omega)} ,
\end{equation*}
and 
\begin{equation*}
\| n_k^\eps \wedge ( \varphi_k^\eps \wedge h_{\rm d} (n_k^\eps) ) \|_{L^2(\Omega)} 
\le \| n_k^\eps \|_{L^6(\Omega)}^2 \| \varphi_k^\eps \|_{L^6(\Omega)} .
\end{equation*}
The above $L^6$ norms are controlled by $H^1$ norms. Thus, for some constant $C$ 
depending on $\| h_{\rm ext} \|_{L^\infty_{t,x}}$, $\| \d_t h_{\rm ext} \|_{L^\infty_{t,x}}$ 
and $\| n_0 \|_{L^\infty((0,\infty),H^1(\Omega))}$: 
\begin{equation} \label{eq:estimI3}
I_3 \le C \| \varphi_k^\eps \|_{L^2(\Omega)} \Big( \| \varphi_k^\eps \|_{H^2(\Omega)} + t \Big) .
\end{equation}

\paragraph{\emph{Estimating $I_4$.}} 
Setting 
$r_k^\eps = \| [P_k,\mathcal{F}(0,\cdot)](n_k^\eps) \|_{L^2(\Omega)}^2$, we have:  
\begin{equation} \label{eq:estimI4}
I_4 \le \| \varphi_k^\eps \|_{L^2(\Omega)}^2 + r_k^\eps, 
\mbox{ and } r_k^\eps \Tend{k}{\infty}0 \mbox{ in } L^1(0,T) 
\mbox{ for all } T>0, \mbox{ with } \eps \mbox{ fixed} .
\end{equation}
This is a consequence of the following lemma, the proof of which 
is postponed to Section~\ref{sec:proofLemmacommut}.

\begin{lemma} \label{lem:commut}
For all $T>0$ and $n \in C([0,T],H^2_N(\Omega)) \cap L^2((0,T),H^3(\Omega))$,
$$[P_k,\mathcal{F}(0,\cdot)](n) \Tend{k}{\infty} 0 
\mbox{ in } L^2((0,T),H^1(\Omega)) .$$
\end{lemma}

\emph{Gathering $L^2$ estimates.} 
Adding \eqref{eq:estimI1}, \eqref{eq:estimI2}, \eqref{eq:estimI3} and \eqref{eq:estimI4}, 
we get 
\begin{equation} \label{eq:estimphikepsL2}
\begin{split}
\frac{\eps}{2} \frac{\rm d}{\rm dt} 
\left( \| \varphi_k^\eps \|_{L^2(\Omega)}^2 \right) 
& + \alpha \| \nabla \varphi_k^\eps \|_{L^2(\Omega)}^2 \\
& \le 
C \| \varphi_k^\eps \|_{H^2(\Omega)} \Big( t + \| \varphi_k^\eps \|_{H^2(\Omega)} 
+ \| \varphi_k^\eps \|_{H^2(\Omega)}^3 + r_k^\eps \Big),
\end{split}
\end{equation}
for some constant depending on the quantities 
$\| h_{\rm ext} \|_{L^\infty_{t,x}}$, 
$\| \d_t h_{\rm ext} \|_{L^\infty_{t,x}}$ and $\| n_0 \|_{L^\infty((0,\infty),H^1(\Omega))}$, 
and $r_k^\eps$ from \eqref{eq:estimI4}. 

\subsubsection{\texorpdfstring{$H^2$}{H2} estimates} 

Next, take the scalar product in $L^2(\Omega)$ of $\Delta^2\varphi_k^\eps$ 
with the first equation in \eqref{eq:phikeps} to get
\begin{equation*}
\frac{\eps}{2} \frac{\rm d}{\rm dt} 
\left( \| \Delta\varphi_k^\eps \|_{L^2(\Omega)}^2 \right) 
+ \alpha \| \nabla\Delta\varphi_k^\eps \|_{L^2(\Omega)}^2 = II_1 + II_2 + II_3 + II_4 , 
\end{equation*}
with
\begin{equation*}
\begin{split}
& II_1 = \Big( \Delta^2\varphi_k^\eps \mid 
m_k^\eps \wedge h_T(t,m_k^\eps) 
- n_k^\eps \wedge h_T(0,n_k^\eps) \Big)_{L^2(\Omega)} , \\
& II_2 = \alpha \Big( \Delta^2\varphi_k^\eps \mid  
|\nabla m_k^\eps|^2 m_k^\eps 
- |\nabla n_k^\eps|^2 n_k^\eps \Big)_{L^2(\Omega)} , \\
& II_3 = - \alpha \Big( \Delta^2\varphi_k^\eps \mid 
m_k^\eps \wedge \Big( m_k^\eps \wedge 
(h_{\rm d}(m_k^\eps) + h_{\rm ext}(t)) \Big) \\
& \hspace{3cm} - n_k^\eps \wedge \Big( n_k^\eps \wedge 
(h_{\rm d}(n_k^\eps) + h_{\rm ext}(0)) \Big) 
\Big)_{L^2(\Omega)} , \\
& II_4 = \Big( \Delta^2\varphi_k^\eps \mid 
[P_k,\mathcal{F}(0,\cdot)](n^\eps) \Big)_{L^2(\Omega)} .
\end{split}
\end{equation*}

\paragraph{\emph{Estimating $II_1$.}} 
Split  
$$II_1 = II_{1,1} + II_{1,2} + II_{1,3} ,$$
with 
\begin{equation*}
\begin{split}
& II_{1,1} = \Big( \Delta^2\varphi_k^\eps \mid m_k^\eps \wedge 
(\Delta\varphi_k^\eps + h_{\rm d}(\varphi_k^\eps)) \Big)_{L^2(\Omega)} , \\
& II_{1,2} = \Big( \Delta^2\varphi_k^\eps \mid \varphi_k^\eps \wedge 
h_T(t,n_k^\eps) \Big)_{L^2(\Omega)} , \\
& II_{1,3} = \Big( \Delta^2\varphi_k^\eps \mid n_k^\eps \wedge 
(h_{\rm ext}(t)-h_{\rm ext}(0)) \Big)_{L^2(\Omega)} .
\end{split}
\end{equation*}

The first term is written 
\begin{equation*}
II_{1,1} = \Big( \Delta^2\varphi_k^\eps \mid n_k^\eps \wedge 
(\Delta\varphi_k^\eps + h_{\rm d}(\varphi_k^\eps)) \Big)_{L^2(\Omega)} 
+ \Big( \Delta^2\varphi_k^\eps \mid \varphi_k^\eps \wedge 
(\Delta\varphi_k^\eps + h_{\rm d}(\varphi_k^\eps)) \Big)_{L^2(\Omega)} .
\end{equation*}

Integrating by parts,
\begin{equation*} 
\begin{split}
& \Big( \Delta^2\varphi_k^\eps \mid n_k^\eps \wedge 
(\Delta\varphi_k^\eps + h_{\rm d}(\varphi_k^\eps)) \Big)_{L^2(\Omega)} = \\
& - \Big( \nabla\Delta\varphi_k^\eps \mid \nabla n_k^\eps \wedge 
(\Delta\varphi_k^\eps + h_{\rm d}(\varphi_k^\eps)) 
+ n_k^\eps \wedge \nabla h_{\rm d}(\varphi_k^\eps) \Big)_{L^2(\Omega)} \\
& \le \eta \| \nabla\Delta\varphi_k^\eps \|_{L^2(\Omega)}^2 
+ \frac{1}{4\eta} \| \nabla n_k^\eps \wedge 
(\Delta\varphi_k^\eps + h_{\rm d}(\varphi_k^\eps)) 
+ n_k^\eps \wedge \nabla h_{\rm d}(\varphi_k^\eps)\|_{L^2(\Omega)}^2 \\
& \le \eta \| \nabla\Delta\varphi_k^\eps \|_{L^2(\Omega)}^2 
+ C_\eta \left( \| n_k^\eps \|_{H^2(\Omega)}^2  
+ \| \nabla\Delta n_k^\eps \|_{L^2(\Omega)}^2 \right) 
\| \varphi_k^\eps \|_{H^2(\Omega)} ,
\end{split}
\end{equation*}
for all $\eta>0$, for some (large) constant $C_\eta$, 
using Sobolev's inequalities.  
>From this, we deduce that for all $\eta>0$, 
there exists $C_\eta>0$, depending only on  
$\| n_0 \|_{L^\infty((0,\infty),H^2(\Omega))}$, such that 
\begin{equation} \label{eq:estimII111}
\begin{split}
\Big( \Delta^2\varphi_k^\eps \mid n_k^\eps \wedge 
& (\Delta\varphi_k^\eps + h_{\rm d}(\varphi_k^\eps)) \Big)_{L^2(\Omega)} 
\le \\
& \eta \| \nabla\Delta\varphi_k^\eps \|_{L^2(\Omega)}^2 
+ C_\eta \left( 1 + \| \nabla\Delta n_k^\eps \|_{L^2(\Omega)}^2 \right) 
\| \varphi_k^\eps \|_{H^2(\Omega)}^2 .
\end{split}
\end{equation}

Integrating by parts again,
\begin{equation*} 
\begin{split}
& \Big( \Delta^2\varphi_k^\eps \mid \varphi_k^\eps \wedge 
(\Delta\varphi_k^\eps + h_{\rm d}(\varphi_k^\eps)) \Big)_{L^2(\Omega)} = \\
& - \Big( \nabla\Delta\varphi_k^\eps \mid \nabla\varphi_k^\eps \wedge 
(\Delta\varphi_k^\eps + h_{\rm d}(\varphi_k^\eps)) 
+ \varphi_k^\eps \wedge \nabla h_{\rm d}(\varphi_k^\eps) \Big)_{L^2(\Omega)} \\
& \le \| \nabla\Delta\varphi_k^\eps \|_{L^2(\Omega)} 
\| \nabla\varphi_k^\eps \|_{L^\infty(\Omega)} \| \Delta\varphi_k^\eps \|_{L^2(\Omega)} \\
& \quad + \| \nabla\Delta\varphi_k^\eps \|_{L^2(\Omega)} \left( 
\| \nabla\varphi_k^\eps \|_{L^4(\Omega)} \| h_{\rm d}(\varphi_k^\eps)) \|_{L^4(\Omega)}
+ \| \varphi_k^\eps \|_{L^4(\Omega)} \| \nabla h_{\rm d}(\varphi_k^\eps)) \|_{L^4(\Omega)}
\right) \\
& \les \| \nabla\Delta\varphi_k^\eps \|_{L^2(\Omega)} 
\left( \| \Delta\nabla\varphi_k^\eps \|_{L^2(\Omega)} + \| \nabla\varphi_k^\eps \|_{L^2(\Omega)} 
\right) \| \varphi_k^\eps \|_{H^2(\Omega)} \\
& \quad + \| \nabla\Delta\varphi_k^\eps \|_{L^2(\Omega)} \| \varphi_k^\eps \|_{H^2(\Omega)}^2,
\end{split}
\end{equation*}
using 
$\| \nabla h_{\rm d}(\varphi_k^\eps) \|_{L^4(\Omega)} 
\les \| h_{\rm d}(\varphi_k^\eps) \|_{H^2(\Omega)} \les \| \varphi_k^\eps \|_{H^2(\Omega)}$. 
Hence, there exists an absolute constant $C>0$, and for all $\eta>0$, there exists $C_\eta>0$ such 
that 
\begin{equation} \label{eq:estimII112}
\begin{split}
\Big( \Delta^2\varphi_k^\eps \mid \varphi_k^\eps \wedge 
& (\Delta\varphi_k^\eps + h_{\rm d}(\varphi_k^\eps)) \Big)_{L^2(\Omega)} \le \\
& \left( \eta + C \| \varphi_k^\eps \|_{H^2(\Omega)} \right) 
\| \nabla\Delta\varphi_k^\eps \|_{L^2(\Omega)}^2 
+ C_\eta \| \varphi_k^\eps \|_{H^2(\Omega)}^4 .
\end{split}
\end{equation}
Summing up \eqref{eq:estimII111} and 
\eqref{eq:estimII112}, one gets $C>0$ and, for all $\eta>0$, a constant $C_\eta>0$ 
(depending on $\| n_0 \|_{L^\infty((0,\infty),H^2(\Omega))}$) such that
\begin{equation} \label{eq:estimII11}
\begin{split}
II_{1,1} \le 
( \eta + 
& C \| \varphi_k^\eps \|_{H^2(\Omega)} ) 
\| \nabla\Delta\varphi_k^\eps \|_{L^2(\Omega)}^2 \\
& + C_\eta \left( 1 + \| \nabla\Delta n_k^\eps \|_{L^2(\Omega)}^2 
+ \| \varphi_k^\eps \|_{H^2(\Omega)}^2 \right) 
\| \varphi_k^\eps \|_{H^2(\Omega)}^2 .
\end{split}
\end{equation}

The second term is 
\begin{equation*}
\begin{split}
II_{1,2} = 
& - \Big( \nabla\Delta\varphi_k^\eps \mid \nabla\varphi_k^\eps \wedge 
(\Delta n_k^\eps + h_{\rm d}(n_k^\eps) + h_{\rm ext}(t)) \Big)_{L^2(\Omega)} \\
& - \Big( \nabla\Delta\varphi_k^\eps \mid \varphi_k^\eps \wedge (\nabla\Delta n_k^\eps 
+ \nabla h_{\rm d}(\varphi_k^\eps) + \nabla h_{\rm ext}(t)) \Big)_{L^2(\Omega)} .
\end{split}
\end{equation*}
Using Sobolev's inequalities again, we have, for all $\eta>0$, 
a constant $C_\eta$ 
(depending on $\| n_0 \|_{L^\infty((0,\infty),H^2(\Omega))}$ 
and $\| h_{\rm ext} \|_{L^\infty_{t,x}}$) such that
\begin{equation} \label{eq:estimII121}
\begin{split}
- 
& \Big( \nabla\Delta\varphi_k^\eps \mid \nabla\varphi_k^\eps \wedge 
(\Delta n_k^\eps + h_{\rm d}(n_k^\eps) + h_{\rm ext}(t)) \Big)_{L^2(\Omega)}
\le \| \nabla\Delta\varphi_k^\eps \|_{L^2(\Omega)} \times \\
& \times \left( \| \nabla\varphi_k^\eps \|_{L^4(\Omega)} 
\| \Delta n_k^\eps + h_{\rm d}(n_k^\eps) \|_{L^4(\Omega)} 
+ \| \nabla\varphi_k^\eps \|_{L^2(\Omega)} \| h_{\rm ext}(t) \|_{L^\infty(\Omega)} \right) \\
& \le C \| \nabla\Delta\varphi_k^\eps \|_{L^2(\Omega)} \| \varphi_k^\eps \|_{H^2(\Omega)} \\ 
& \qquad \qquad \left( \| \nabla\Delta n_k^\eps \|_{L^2(\Omega)} + \| n_k^\eps \|_{H^2(\Omega)} + 
\| h_{\rm ext}(t) \|_{L^\infty(\Omega)} \right) \\
& \le \eta \| \nabla\Delta\varphi_k^\eps \|_{L^2(\Omega)}^2 + 
C_\eta \left( 1 + \| \nabla\Delta n_k^\eps \|_{L^2(\Omega)}^2 \right) 
\| \varphi_k^\eps \|_{H^2(\Omega)}^2.
\end{split}
\end{equation}
In the same way, for all $\eta>0$, 
there is $C_\eta>0$ (depending on $\| n_0 \|_{L^\infty((0,\infty),H^2(\Omega))}$ and 
$\| \nabla h_{\rm ext} \|_{L^\infty_{t,x}}$) such that 
\begin{equation} \label{eq:estimII122}
\begin{split}
& - \Big( \nabla\Delta\varphi_k^\eps \mid \varphi_k^\eps \wedge (\nabla\Delta n_k^\eps 
+ \nabla h_{\rm d}(\varphi_k^\eps) + \nabla h_{\rm ext}(t)) \Big)_{L^2(\Omega)} \le \\
& \le \eta \| \nabla\Delta\varphi_k^\eps \|_{L^2(\Omega)}^2 + 
C_\eta \left( 1 + \| \nabla\Delta n_k^\eps \|_{L^2(\Omega)}^2 \right) 
\| \varphi_k^\eps \|_{H^2(\Omega)}^2.
\end{split}
\end{equation}
Summing up \eqref{eq:estimII121} and \eqref{eq:estimII122}, we get, for all $\eta>0$, 
a constant $C_\eta>0$ (depending on $\| h_{\rm ext} \|_{L^\infty_{t,x}}$, 
$\| \nabla h_{\rm ext} \|_{L^\infty_{t,x}}$ and 
$\| n_0 \|_{L^\infty((0,\infty),H^2(\Omega))}$) such that
\begin{equation} \label{eq:estimII12}
II_{1,2} \le \eta \| \nabla\Delta\varphi_k^\eps \|_{L^2(\Omega)}^2 
+ C_\eta \left( 1 + \| \nabla\Delta n_k^\eps \|_{L^2(\Omega)}^2 \right) 
\| \varphi_k^\eps \|_{H^2(\Omega)}^2.
\end{equation}

The third term is 
\begin{equation*}
II_{1,3} = \Big( \Delta^2\varphi_k^\eps \mid n_k^\eps \wedge 
(h_{\rm ext}(t)-h_{\rm ext}(0)) \Big)_{L^2(\Omega)} .
\end{equation*}
Integrating two times by parts, it is easily estimated, thanks to a constant 
$C$ depending on $\| \d_t h_{\rm ext} \|_{L^\infty_t W^{2,\infty}_x}$, as
\begin{equation} \label{eq:estimII13}
II_{1,3} \le C t \| n_k^\eps \|_{H^2(\Omega)} \| \varphi_k^\eps \|_{H^2(\Omega)} .
\end{equation} 

This gives finally, summing up \eqref{eq:estimII11}, \eqref{eq:estimII12} and \eqref{eq:estimII13}: 
there is $C>0$, and for all $\eta>0$, there is $C_\eta>0$ (depending on $\eta$, 
$\| n_0 \|_{L^\infty((0,\infty),H^2(\Omega))}$, $\| h_{\rm ext} \|_{L^\infty_t W^{1,\infty}_x}$ and 
$\| \d_t h_{\rm ext} \|_{L^\infty_t W^{2,\infty}_x}$), such that
\begin{equation} \label{eq:estimII1}
\begin{split}
II_1 \le \, 
& \Big( \eta + C \| \varphi_k^\eps \|_{H^2(\Omega)} \Big) 
\| \nabla\Delta\varphi_k^\eps \|_{L^2(\Omega)}^2 \\
& + C_\eta \Big( \left( 1 + \| \nabla\Delta n_k^\eps \|_{L^2(\Omega)}^2 
+ \| \varphi_k^\eps \|_{H^2(\Omega)}^2 \right) \| \varphi_k^\eps \|_{H^2(\Omega)}^2 
+ t \| \varphi_k^\eps \|_{H^2(\Omega)} \Big) .
\end{split}
\end{equation}

\paragraph{\emph{Estimating $II_2$.}} 
Split  
$$II_2 = II_{2,1} + II_{2,2} ,$$
with 
\begin{equation*}
\begin{split}
II_{2,1} 
& = \Big( \Delta^2\varphi_k^\eps \mid 
( |\nabla m_k^\eps|^2-|\nabla n_k^\eps|^2 ) 
( n_k^\eps + \varphi_k^\eps ) \Big)_{L^2(\Omega)} \\
& = \Big( \Delta^2\varphi_k^\eps \mid 
\nabla\varphi_k^\eps \cdot ( \nabla \varphi_k^\eps 
+ 2\nabla n_k^\eps ) m_k^\eps \Big)_{L^2(\Omega)} , \\
II_{2,2} 
& = \Big( \Delta^2\varphi_k^\eps \mid 
|\nabla n_k^\eps|^2 (m_k^\eps-n_k^\eps) \Big)_{L^2(\Omega)} 
= \Big( \Delta^2\varphi_k^\eps \mid 
|\nabla n_k^\eps|^2 \varphi_k^\eps \Big)_{L^2(\Omega)} .
\end{split}
\end{equation*}
Then, using in particular the Sobolev inequality from 
Lemma~\ref{lem:sobolev}
$$
\| \nabla n_k^\eps \|_{L^\infty(\Omega)} 
\les \| \nabla n_k^\eps \|_{L^2(\Omega)} 
+ \| \nabla \Delta n_k^\eps \|_{L^2(\Omega)} ,
$$
we get:
\begin{equation} \label{eq:estimII21}
\begin{split} 
& II_{2,1} = 
- \Big( \nabla\Delta\varphi_k^\eps \mid 
\Delta\varphi_k^\eps \cdot ( \nabla \varphi_k^\eps + 2\nabla n_k^\eps ) 
( n_k^\eps + \varphi_k^\eps ) \Big)_{L^2(\Omega)} \\
& \qquad\quad - \Big( \nabla\Delta\varphi_k^\eps \mid 
\nabla\varphi_k^\eps \cdot ( \Delta \varphi_k^\eps + 2\Delta n_k^\eps ) 
( n_k^\eps + \varphi_k^\eps ) \Big)_{L^2(\Omega)} \\
& \qquad\quad - \Big( \nabla\Delta\varphi_k^\eps \mid 
\nabla\varphi_k^\eps \cdot ( \nabla \varphi_k^\eps + 2\nabla n_k^\eps ) 
( \nabla n_k^\eps + \nabla\varphi_k^\eps ) \Big)_{L^2(\Omega)} \\
& \le \| \nabla\Delta\varphi_k^\eps \|_{L^2(\Omega)} 
\| \Delta\varphi_k^\eps \|_{L^2(\Omega)} \times \\ 
& \qquad \times 
\Big( \| \nabla\varphi_k^\eps \|_{L^\infty(\Omega)} 
+ 2 \| \nabla n_k^\eps \|_{L^\infty(\Omega)} \Big)
\Big( \| n_k^\eps \|_{L^\infty(\Omega)} + \| \varphi_k^\eps \|_{L^\infty(\Omega)} \Big) \\
& \quad + \| \nabla\Delta\varphi_k^\eps \|_{L^2(\Omega)} 
\| \nabla\varphi_k^\eps \|_{L^4(\Omega)} \times \\ 
& \qquad \times 
\Big( \| \Delta\varphi_k^\eps \|_{L^4(\Omega)} 
+ 2 \| \Delta n_k^\eps \|_{L^4(\Omega)} \Big)
\Big( \| n_k^\eps \|_{L^\infty(\Omega)} 
+ \| \varphi_k^\eps \|_{L^\infty(\Omega)} \Big) \\
& \quad + \| \nabla\Delta\varphi_k^\eps \|_{L^2(\Omega)} 
\| \nabla\varphi_k^\eps \|_{L^6(\Omega)} \times \\ 
& \qquad \times 
\Big( \| \nabla\varphi_k^\eps \|_{L^6(\Omega)} 
+ 2 \| \nabla n_k^\eps \|_{L^6(\Omega)} \Big)
\Big( \| \nabla n_k^\eps \|_{L^6(\Omega)} 
+ \| \nabla\varphi_k^\eps \|_{L^6(\Omega)} \Big) \\
& \les \| \nabla\Delta\varphi_k^\eps \|_{L^2(\Omega)} 
\| \varphi_k^\eps \|_{H^2(\Omega)} 
\Big( \| n_k^\eps \|_{H^2(\Omega)} 
+ \| \varphi_k^\eps \|_{H^2(\Omega)} \Big) \times \\ 
& \qquad \times 
\Big( \| \nabla\varphi_k^\eps \|_{L^2(\Omega)} 
+ \| \nabla\Delta\varphi_k^\eps \|_{L^2(\Omega)} 
+ \| \nabla n_k^\eps \|_{L^2(\Omega)} 
+ \| \nabla\Delta n_k^\eps \|_{L^2(\Omega)} \Big) \\
& \quad + \| \nabla\Delta\varphi_k^\eps \|_{L^2(\Omega)} 
\| \varphi_k^\eps \|_{H^2(\Omega)} 
\Big( \| n_k^\eps \|_{H^2(\Omega)} 
+ \| \varphi_k^\eps \|_{H^2(\Omega)} \Big) \times \\
& \qquad \times 
\Big( \| \Delta\varphi_k^\eps \|_{L^2(\Omega)} 
+ \| \nabla\Delta\varphi_k^\eps \|_{L^2(\Omega)} 
+ \| \Delta n_k^\eps \|_{L^2(\Omega)} 
+ \| \nabla\Delta n_k^\eps \|_{L^2(\Omega)} \Big) \\
& \quad + \| \nabla\Delta\varphi_k^\eps \|_{L^2(\Omega)} 
\| \varphi_k^\eps \|_{H^2(\Omega)} 
\Big( \| n_k^\eps \|_{H^2(\Omega)} 
+ \| \varphi_k^\eps \|_{H^2(\Omega)} \Big)^2 \\
& \les \| \nabla\Delta\varphi_k^\eps \|_{L^2(\Omega)} 
\| \varphi_k^\eps \|_{H^2(\Omega)} 
\Big( \| n_k^\eps \|_{H^2(\Omega)} 
+ \| \varphi_k^\eps \|_{H^2(\Omega)} \Big) \times \\
& \qquad \times 
\Big( \| \varphi_k^\eps \|_{H^2(\Omega)} 
+ \| \nabla\Delta\varphi_k^\eps \|_{L^2(\Omega)} 
+ \| n_k^\eps \|_{H^2(\Omega)} 
+ \| \nabla \Delta n_k^\eps \|_{L^2(\Omega)} \Big) \\
& \les \Big( \eta + \| \varphi_k^\eps \|_{H^2(\Omega)}
\Big( \| n_k^\eps \|_{H^2(\Omega)} + 
\| \varphi_k^\eps \|_{H^2(\Omega)} \Big) \Big) 
\| \nabla\Delta\varphi_k^\eps \|_{L^2(\Omega)}^2 \\ 
& \qquad + C_\eta \| \varphi_k^\eps \|_{H^2(\Omega)}^2 
\Big( \| \varphi_k^\eps \|_{H^2(\Omega)} 
+ \| n_k^\eps \|_{H^3(\Omega)} 
+ \| \nabla \Delta n_k^\eps \|_{L^2(\Omega)} \Big)^2 \times \\
& \qquad \times 
\Big( \| n_k^\eps \|_{H^2(\Omega)} 
+ \| \varphi_k^\eps \|_{H^2(\Omega)} \Big)^2 ,
\end{split}
\end{equation}
for all $\eta>0$, for some $C_\eta>0$.

Also, for all $\eta>0$, there is $C_\eta>0$ such that
\begin{equation} \label{eq:estimII22}
\begin{split}
II_{2,2} 
& = - \Big( \nabla\Delta\varphi_k^\eps \mid 
2\nabla n_k^\eps \Delta n_k^\eps \varphi_k^\eps \Big)_{L^2(\Omega)} 
- \Big( \nabla\Delta\varphi_k^\eps \mid 
|\nabla n_k^\eps|^2 \nabla\varphi_k^\eps \Big)_{L^2(\Omega)} \\
& \le \| \nabla\Delta\varphi_k^\eps \|_{L^2(\Omega)} 
\Big( 2 \| \nabla n_k^\eps \|_{L^\infty(\Omega)} \| \Delta n_k^\eps \|_{L^2(\Omega)} 
\| \varphi_k^\eps \|_{L^\infty(\Omega)} \\
& \qquad \qquad \qquad \qquad \qquad \qquad + \| \nabla n_k^\eps \|_{L^6(\Omega)}^2 
\| \nabla\varphi_k^\eps \|_{L^6(\Omega)} \Big) \\
& \les \eta \| \nabla\Delta\varphi_k^\eps \|_{L^2(\Omega)}^2 \\ 
& \qquad + C_\eta \Big( \| n_k^\eps \|_{H^2(\Omega)}^2 
+ \| \nabla \Delta n_k^\eps \|_{L^2(\Omega)}^2 + 1 \Big) 
\| n_k^\eps \|_{H^2(\Omega)}^2 \| \varphi_k^\eps \|_{H^2(\Omega)}^2 .
\end{split}
\end{equation}
Summing up \eqref{eq:estimII21} and \eqref{eq:estimII22}, we get: 
there is $C>0$, and for all $\eta>0$, there is $C_\eta>0$ (with $C$ 
and $C_\eta$ depending on 
$\| n_0 \|_{L^\infty((0,\infty),H^2(\Omega))}$) such that 
\begin{equation} \label{eq:estimII2}
\begin{split}
II_2 
& \le C \Big( \eta + \| \varphi_k^\eps \|_{H^2(\Omega)}
\Big( 1 + \| \varphi_k^\eps \|_{H^2(\Omega)} \Big) \Big) 
\| \nabla\Delta\varphi_k^\eps \|_{L^2(\Omega)}^2 \\ 
& \qquad + C_\eta \| \varphi_k^\eps \|_{H^2(\Omega)}^2 
\Big( 1 + \| \varphi_k^\eps \|_{H^2(\Omega)}^2 \Big) 
\Big( 1 + \| \varphi_k^\eps \|_{H^2(\Omega)}^2 
+ \| \nabla \Delta n_k^\eps \|_{L^2(\Omega)}^2 \Big).
\end{split}
\end{equation}

\paragraph{\emph{Estimating $II_3$.}} 
Now, 
$$II_3 = II_{3,1} + II_{3,2} ,$$
with
\begin{equation*}
\begin{split}
II_{3,1} & = - \alpha \Big( \Delta^2\varphi_k^\eps \mid 
m_k^\eps \wedge \Big( m_k^\eps \wedge h_{\rm d}(m_k^\eps) \Big) 
- n_k^\eps \wedge \Big( n_k^\eps \wedge h_{\rm d}(n_k^\eps) \Big) \Big)_{L^2(\Omega)}, \\
II_{3,2} & = - \alpha \Big( \Delta^2\varphi_k^\eps \mid 
m_k^\eps \wedge \Big( m_k^\eps \wedge h_{\rm ext}(t) \Big) 
- n_k^\eps \wedge \Big( n_k^\eps \wedge h_{\rm ext}(0) \Big) \Big)_{L^2(\Omega)}.  
\end{split}
\end{equation*}
Concerning $II_{3,1}$, first write $m_k^\eps = n_k^\eps + \varphi_k^\eps$, then integrate 
once by parts, so that $II_{3,1}$ takes the form of a  $L^2$ scalar product between 
$\nabla\Delta\varphi_k^\eps$ and a sum of terms $\nabla(abc)$, where $a$, $b$, $c$ may be 
$n_k^\eps$ (or $h_{\rm d}(n_k^\eps)$) or $\varphi_k^\eps$ (or $h_{\rm d}(\varphi_k^\eps)$), 
and at least one of them is $n_k^\eps$ (or $h_{\rm d}(n_k^\eps)$). Estimating each of $a$, 
$b$, $c$ and their gradients in $L^6$, one gets: for all $\eta>0$, there is $C_\eta>0$ such 
that 
\begin{equation} \label{eq:estimII31}
II_{3,1} \le \eta \| \nabla\Delta\varphi_k^\eps \|_{L^2(\Omega)}^2 
+ C_\eta \Big( 1 + \| n_k^\eps \|_{H^2(\Omega)}^2 + \| \varphi_k^\eps \|_{H^2(\Omega)}^2 \Big)^2 
\| \varphi_k^\eps \|_{H^2(\Omega)}^2 .
\end{equation}
Then, split $II_{3,2}$, 
\begin{equation*} 
\begin{split}
II_{3,2} 
& = - \alpha \Big( \Delta^2\varphi_k^\eps \mid m_k^\eps \wedge \Big( m_k^\eps \wedge 
( h_{\rm ext}(t) - h_{\rm ext}(0) ) \Big) \Big)_{L^2(\Omega)} \\
& \quad - \alpha \Big( \Delta^2\varphi_k^\eps \mid 
m_k^\eps \wedge \Big( m_k^\eps \wedge h_{\rm ext}(0) \Big)
- n_k^\eps \wedge \Big( n_k^\eps \wedge h_{\rm ext}(0) \Big) \Big)_{L^2(\Omega)}.
\end{split}
\end{equation*} 

The second term is estimated as $II_{3,1}$. The first one is split into a sum 
involving $n_k^\eps \wedge \Big( n_k^\eps \wedge ( h_{\rm ext}(t) - h_{\rm ext}(0) ) \Big)$, 
and products of $h_{\rm ext}(t)$ with two terms, one of them being $\varphi_k^\eps$, and the 
other, $\varphi_k^\eps$ or $n_k^\eps$. This leads to: for all $\eta>0$, there is $C_\eta>0$ 
(also depending on $h_{\rm ext}$) such that 
\begin{equation} \label{eq:estimII32}
\begin{split}
II_{3,2} \le 
& \eta \| \nabla\Delta\varphi_k^\eps \|_{L^2(\Omega)}^2 
+ C_\eta \Big( 
\Big( 1 + \| n_k^\eps \|_{H^2(\Omega)}^2 + \| \varphi_k^\eps \|_{H^2(\Omega)}^2 \Big) 
\| \varphi_k^\eps \|_{H^2(\Omega)}^2 \\
& \qquad \qquad \qquad \qquad \qquad \qquad \qquad + t \| n_k^\eps \|_{H^2(\Omega)}^2 
\| \varphi_k^\eps \|_{H^2(\Omega)} \Big) .
\end{split}
\end{equation}
Finally, summing up \eqref{eq:estimII31} and \eqref{eq:estimII32}, 
we have: for all $\eta>0$, there is $C_\eta>0$ (depending on 
$\| n_0 \|_{L^\infty((0,\infty),H^2(\Omega))}$) such that
\begin{equation} \label{eq:estimII3}
\begin{split}
II_3 \le 
& \eta \| \nabla\Delta\varphi_k^\eps \|_{L^2(\Omega)}^2 \\
& + C_\eta \Big( 
\Big( 1 + \| \varphi_k^\eps \|_{H^2(\Omega)}^2 \Big)^2 
\| \varphi_k^\eps \|_{H^2(\Omega)}^2 
+ t \| \varphi_k^\eps \|_{H^2(\Omega)} \Big) .
\end{split}
\end{equation}

\paragraph{\emph{Estimating $II_4$.}} 
Integrating once by parts, we get
$$II_4 = - \Big( \nabla\Delta\varphi_k^\eps \mid 
\nabla[P_k,\mathcal{F}(0,\cdot)](n^\eps) \Big)_{L^2(\Omega)} .$$
Thus, for all $\eta>0$, there exists $C_\eta>0$ such that 
\begin{equation} \label{eq:estimII4}
II_4 \le \eta \| \nabla\Delta\varphi_k^\eps \|_{L^2(\Omega)}^2 + C_\eta {r_{k,1}^\eps},
\end{equation}
with 
$$r_{k,1}^\eps = \| \nabla[P_k,\mathcal{F}(0,\cdot)](n^\eps) \|_{L^2(\Omega)}^2 \Tend{k}{\infty} 0 
\mbox{ in } L^\infty(0,T) \mbox{ for all } T>0, \mbox{ with } \eps \mbox{ fixed} ,$$
thanks to Lemma \ref{lem:commut}.

\subsubsection{Conclusion}  

>From \eqref{eq:estimII1}, \eqref{eq:estimII2}, \eqref{eq:estimII3} and \eqref{eq:estimII4}, 
we deduce that there is a constant $C>0$ (depending on $\| n_0 \|_{L^\infty((0,\infty),H^3(\Omega))}$), 
and for all $\eta>0$, there is $C_\eta>0$ (depending on $\eta$, 
$\| n_0 \|_{L^\infty((0,\infty),H^3(\Omega))}$, $\| h_{\rm ext} \|_{L^\infty_t W^{1,\infty}_x}$ and 
$\| \d_t h_{\rm ext} \|_{L^\infty_t W^{2,\infty}_x}$), such that 
\begin{equation} \label{eq:estimDeltaphikepsL2}
\begin{split}
\frac{\eps}{2} \frac{\rm d}{\rm dt} 
& \left( \| \Delta\varphi_k^\eps \|_{L^2(\Omega)}^2 \right) \\ 
+ & \Big( \alpha - C ( \eta + \| \varphi_k^\eps \|_{H^2(\Omega)} 
( 1 + \| \varphi_k^\eps \|_{H^2(\Omega)} ) ) \Big) 
\| \nabla\Delta\varphi_k^\eps \|_{L^2(\Omega)}^2 \\
& \le C_\eta \Big( \| \varphi_k^\eps \|_{H^2(\Omega)}^2 
( 1 + \| \varphi_k^\eps \|_{H^2(\Omega)}^2 ) 
( 1 + \| \varphi_k^\eps \|_{H^2(\Omega)}^2 
+ \| \nabla \Delta n_k^\eps \|_{L^2(\Omega)}^2) \\ 
& \qquad\qquad\qquad\qquad\qquad\qquad
+ t \| \varphi_k^\eps \|_{H^2(\Omega)} + r_{k,1}^\eps \Big) .
\end{split} 
\end{equation}
Sum up \eqref{eq:estimphikepsL2} and \eqref{eq:estimDeltaphikepsL2}, to get:  
there is a constant $C>0$ (depending on $\| n_0 \|_{L^\infty((0,\infty),H^3(\Omega))}$), and 
for all $\eta>0$, there is $C_\eta>0$ (depending on $\eta$, 
$\| n_0 \|_{L^\infty((0,\infty),H^2(\Omega))}$, $\| h_{\rm ext} \|_{L^\infty_t W^{1,\infty}_x}$ and 
$\| \d_t h_{\rm ext} \|_{L^\infty_t W^{2,\infty}_x}$),
such that 
\begin{equation} \label{eq:estimphikepsH2}
\begin{split}
\frac{\eps}{2} \frac{\rm d}{\rm dt} 
& \left( \| \varphi_k^\eps \|_{H^2(\Omega)}^2 \right) 
+ \Big( \alpha - C ( \eta + \| \varphi_k^\eps \|_{H^2(\Omega)} 
( 1 + \| \varphi_k^\eps \|_{H^2(\Omega)} ) ) \Big)  
\| \nabla\Delta\varphi_k^\eps \|_{H^2(\Omega)}^2 \\
& \le C_\eta \Big( \| \varphi_k^\eps \|_{H^2(\Omega)}^2 
( 1 + \| \varphi_k^\eps \|_{H^2(\Omega)}^2 ) 
( 1 + \| \varphi_k^\eps \|_{H^2(\Omega)}^2 
+ \| \nabla \Delta n_k^\eps \|_{L^2(\Omega)}^2) \\ 
& \qquad\qquad\qquad\qquad\qquad\qquad
+ t \| \varphi_k^\eps \|_{H^2(\Omega)} + \tilde{r}_k^\eps \Big) ,
\end{split}
\end{equation}
with $\tilde{r}_k^\eps = \| [P_k,\mathcal{F}(0,\cdot)](n^\eps) \|_{H^1(\Omega)}^2 \Tend{k}{\infty} 0$  
in $L^\infty(0,T)$ for all $T>0$, with $\eps$ fixed.

Now, apply the following Gronwall lemma (the proof of which is postponed to 
Section~\ref{sec:proofLemmagronwall}).

\begin{lemma} \label{lem:gronwall}
There is a constant $K>0$ (depending on $n_0$ and $h_{\rm ext}$) 
such that, for all $c\in(0,1/K)$, setting 
$t_\eps = c\eps\ln(1/\eps)$, 
there is $\eps_0 = \eps_0(\alpha,c,K)$ such that \eqref{eq:estimphikepsH2} implies: 
\begin{equation*}
\begin{split}
\forall\eps & \in (0,\eps_0] , \quad 
\exists \underline{k}(\eps)\in\N^\star, \quad \forall k \ge \underline{k}(\eps) , \\ 
& \sup_{[0,t_\eps]} \| \varphi_k^\eps \|_{H^2(\Omega)}^2 \le \left(  
\frac{\eps^{1-cK}}{K} + K \| \tilde{r}_k^\eps \|_{L^1(0,t_{\eps_0})} \eps^{-1-cK} \right) 
e^{K \| \nabla \Delta P_k n_0 \|_{L^2((0,\infty)\times\Omega)}^2}.
\end{split}
\end{equation*}
\end{lemma}

\subsubsection{Passing to the limit \texorpdfstring{$k\rightarrow\infty$}{k to infinity}}

For each $\eps \in (0,\eps_0]$ fixed, 
by Lemma \ref{lem:gronwall}, the sequence 
$(\varphi_k^\eps)_{k\in\N^\star}$ is bounded 
in $L^\infty((0,t_\eps),H^2(\Omega))$. 
Equation \eqref{eq:phikeps} then implies that the sequence $(\partial_t\varphi_k^\eps)_{k\in\N^\star}$ is bounded in
$L^\infty((0,t_\eps),L^2(\Omega))$. Furthermore, \eqref{eq:estimphikepsH2} shows that 
$(\varphi_k^\eps)_{k\in\N^\star}$ is also bounded in $L^2((0,t_\eps),H^3(\Omega))$. Aubin's Lemma 
(see \cite{A63}, \cite{S87}) then implies that there is a 
subsequence of $(\varphi_k^\eps)_{k\in\N^\star}$ converging in $L^2((0,t_\eps),H^2(\Omega))$ towards some 
$\varphi^\eps$. 

Up to a subsequence, we may assume that $(\partial_t\varphi_k^\eps)_{k\in\N^\star}$ also converges 
weakly in $L^2((0,t_\eps),L^2(\Omega))$ towards $\partial_t\varphi^\eps$. As $k$ goes to $\infty$, 
$P_k n^\eps$ converges towards 
$n^\eps$ in $C([0,t_\eps],H^2(\Omega)) \cap H^1((0,t_\eps),L^2(\Omega))$. 
Thus, $(m_k^\eps)_{k\in\N^\star}$ converges 
towards some $m^\eps$ in $L^2((0,t_\eps),H^2(\Omega))$, 
with $(\partial_t m_k^\eps)_{k\in\N^\star}$  
converging weakly in $L^2((0,t_\eps),L^2(\Omega))$ towards 
$\partial_t m^\eps$. This is enough to 
pass to the limit in \eqref{eq:mkeps}, so that $m^\eps$ is 
solution to \eqref{eq:llparab}. With $\eps$ fixed, showing that 
$m^\eps$ is continuous in time with values in $H^2$ is standard, 
as well as uniqueness and stability properties: see \cite{CF01}, 
or \cite{AB09}. 
 
Finally, passing to the limit in Lemma~\ref{lem:gronwall} yields: 
$$
\sup_{[0,t_\eps]} \| \varphi^\eps \|_{H^2(\Omega)}^2 
\le \frac{\eps^{1-cK}}{K} 
e^{K \| \nabla \Delta n_0 \|_{L^2((0,\infty)\times\Omega)}^2},
$$
which we write
\begin{equation} \label{eq:estimphiepsH2}
\sup_{[0,t_\eps]} \| \varphi^\eps \|_{H^2(\Omega)}^2 
\le K'\eps^{1-cK}.
\end{equation}

\subsection{Second step: following the slow dynamics 
after \texorpdfstring{$t_\eps$}{t eps}} \label{sec:secondstep}

>From the local-in-time existence result, we know that, 
for each $\eps\in(0,\eps_0)$, 
there is $t^\eps>t_\eps$ such that $m^\eps$ exists, 
as a solution to \eqref{eq:llparab}, 
in $C([0,t^\eps],H^2(\Omega)) \cap L^2((0,t^\eps),H^3(\Omega))$. We shall show, \emph{via} 
\emph{a priori} estimates, that $t^\eps \ge T$ (possibly reducing $\eps_0$). 

>From \eqref{eq:caractmeq}, and with $\mathcal{F}$ from \eqref{eq:defF}, we deduce that, 
on $[0,T]\times\Omega$,  
$$
\eps \d_t m_{\rm eq} - \alpha \Delta m_{\rm eq} = \mathcal{F}(t,m_{\rm eq}) + \eps \d_t m_{\rm eq} .
$$
Substracting to \eqref{eq:llparab}, we get (on $[0,t^\eps]\times\Omega$): 
\begin{equation}
\label{eq:evolm-meq}
\left\{  
\begin{array}{ll}
& (\eps \d_t - \alpha \Delta) (m^\eps-m_{\rm eq}) = \Big( \mathcal{L}(m_{\rm eq})  
+ \mathcal{R}(m_{\rm eq}) \Big) (m^\eps-m_{\rm eq}) + \eps \d_t m_{\rm eq} , \\
& \dn (m^\eps-m_{\rm eq})_{|_{\d\Omega}} = 0 ,
\end{array} \right.
\end{equation}
and we consider the associated Cauchy problem with data given 
at time $t_\eps$. The data at time $t_\eps = c\eps\ln(1/\eps)$ 
satisfy (using \eqref{eq:estimphiepsH2} and \eqref{eq:n0cv}):
\begin{equation} \label{eq:deltaepsini}
\begin{split}
\| (m^\eps-m_{\rm eq})(t_\eps) \|_{H^2(\Omega)} 
& \le
\| (m^\eps-n^\eps)(t_\eps) \|_{H^2(\Omega)} + 
\| n^\eps(t_\eps)-m_{\rm eq}(0) \|_{H^2(\Omega)} \\
& \qquad\qquad\qquad 
+ \| m_{\rm eq}(0)-m_{\rm eq}(t_\eps) \|_{H^2(\Omega)} \\ 
& \le K'\eps^{1-cK} 
+ \| n_0(c\ln(1/\eps))-m_{\rm eq}(0) \|_{H^2(\Omega)} \\
& \qquad\qquad\qquad 
+ \| m_{\rm eq}(0)-m_{\rm eq}(t_\eps) \|_{H^2(\Omega)} 
\Tend{\eps}{0} 0 . 
\end{split}
\end{equation}
Here, for all $\delta \in H^2(\Omega)$ and $t \in [0,T]$, 
\begin{equation} \label{eq:defLeq}
\begin{split}
{\mathcal L}(t,m_{\rm eq}(t)) \, \delta = 
& \alpha |\nabla m_{\rm eq}(t)|^2 \delta 
+ 2\alpha \Big( \nabla m_{\rm eq}(t)\cdot\nabla\delta \Big) m_{\rm eq}(t) \\
& + \delta \wedge h_T(t,m_{\rm eq}(t)) + m_{\rm eq}(t) \wedge 
\Big( \Delta\delta + h_{\rm d}(\delta) \Big) \\ 
& - \alpha \delta \wedge 
\Big( m_{\rm eq}(t) \wedge \Big( h_{\rm d}(m_{\rm eq}(t)) + h_{\rm ext}(t) \Big) \Big) \\
& - \alpha m_{\rm eq}(t) \wedge 
\Big( \delta \wedge \Big( h_{\rm d}(m_{\rm eq}(t)) + h_{\rm ext}(t) \Big) \Big) \\
& - \alpha m_{\rm eq}(t) \wedge \Big( m_{\rm eq}(t) \wedge h_{\rm d}(\delta) \Big) ,
\end{split}
\end{equation}
and 
\begin{equation} \label{eq:defR}
\begin{split}
{\mathcal R}(t,m_{\rm eq}(t)) & \, (\delta) =  
2 \alpha \Big( \nabla m_{\rm eq}(t) \cdot \nabla\delta \Big) \delta 
+ \alpha |\nabla\delta|^2 \delta 
+ \delta \wedge \Big( \Delta\delta + h_{\rm d}(\delta) \Big) \\
& - \alpha \delta \wedge 
\Big( \delta \wedge \Big( h_{\rm d}(m_{\rm eq}(t)) + h_{\rm ext}(t) \Big) \Big) 
- \alpha \delta \wedge \Big( m_{\rm eq}(t) \wedge h_{\rm d}(\delta) \Big) \\
& - \alpha m_{\rm eq}(t) \wedge \Big( \delta \wedge h_{\rm d}(\delta) \Big) 
- \alpha \delta \wedge \Big( \delta \wedge h_{\rm d}(\delta) \Big) .
\end{split}
\end{equation}
In the sequel, we consider 
$$\delta^\eps := m^\eps-m_{\rm eq} 
\in C([0,t^\eps],H^2(\Omega)) \cap L^2((0,t^\eps),H^3(\Omega)),$$
and we simply prove that in ($H^2$) energy estimates, the term due to the residual 
$\mathcal{R}(m_{\rm eq})\delta^\eps$ is dominated by the terms due to $\alpha\Delta
\delta^\eps$ and to the linear term $\mathcal{L}(m_{\rm eq})\delta^\eps$. 
We thus come back to the Galerkine approximation $\delta^\eps_k$ of $\delta^\eps$, 
as in Paragraph~\ref{sec:galerkine}. Take the $L^2(\Omega)$ scalar product 
of the equations with $\delta^\eps_k$ and $\Delta^2\delta^\eps_k$ 
and integrate by parts. Estimating $( \delta^\eps_k \mid \mathcal{R}(m_{\rm eq})
(\delta^\eps_k) )_{L^2(\Omega)}$ is straightforward. Due to the continuity properties of 
$h_{\rm d}$ on Sobolev spaces, $( \Delta^2\delta^\eps_k \mid \mathcal{R}(m_{\rm eq})
(\delta^\eps_k) )_{L^2(\Omega)}$ produces three kinds of terms. 
Dropping the exponent $\eps$ and subscript $k$ 
(and using the notation $L(v_1,\dots,v_n)$ for any $n$-linear 
application), we examine each of them. 
\paragraph{\emph{From 
$\delta \wedge \Big( \delta \wedge h_{\rm d}(\delta) \Big)$.}} 
We have 
$$( \Delta^2\delta \mid L(\delta,\delta,\delta) )_{L^2(\Omega)} 
= ( \Delta\delta \mid \Delta L(\delta,\delta,\delta) )_{L^2(\Omega)} 
\le \| \Delta\delta \|_{L^2(\Omega)} \| \Delta L(\delta,\delta,\delta) \|_{L^2(\Omega)},$$
which is bounded from above by $C \| \delta \|_{H^2(\Omega)}^4$, 
since $H^2(\Omega)$ is an algebra. 

In the same way, the terms of the form 
$( \Delta^2\delta \mid L(\delta,\delta) )_{L^2(\Omega)}$ are controlled by 
$\| \delta \|_{H^2(\Omega)}^3$. This rules out the terms from 
$\delta \wedge h_{\rm d}(\delta)$, 
$\delta \wedge \Big( \delta \wedge \Big( h_{\rm d}(m_{\rm eq}(t)) + h_{\rm ext}(t) \Big) \Big)$, 
$\delta \wedge \Big( m_{\rm eq}(t) \wedge h_{\rm d}(\delta) \Big)$ 
and $m_{\rm eq}(t) \wedge \Big( \delta \wedge h_{\rm d}(\delta) \Big)$.
\paragraph{\emph{From $|\nabla\delta|^2 \delta$.}}  
Write
\begin{equation*}
\begin{split}
( \Delta^2\delta \mid 
L(\nabla\delta,\nabla\delta,& \delta) )_{L^2(\Omega)} = \\
& -( \nabla\Delta\delta \mid 
L(\nabla\delta,\nabla\delta,\nabla\delta) )_{L^2(\Omega)} 
-( \nabla\Delta\delta \mid 
\tilde{L}(\Delta\delta,\nabla\delta,\delta) )_{L^2(\Omega)}.
\end{split}
\end{equation*}
Then, 
\begin{equation*}
\begin{split}
|( \nabla\Delta\delta \mid 
L(\nabla\delta,\nabla\delta,\nabla\delta) )_{L^2(\Omega)}| 
& \le \| \nabla\Delta\delta \|_{L^2(\Omega)} 
\| L(\nabla\delta,\nabla\delta,\nabla\delta) \|_{L^2(\Omega)} \\
& \le C \| \nabla\Delta\delta \|_{L^2(\Omega)} 
\| \nabla\delta \|_{L^6(\Omega)}^3, 
\end{split}
\end{equation*}
and by Sobolev's inequalities, $\| \nabla\delta \|_{L^6(\Omega)}$ is controlled by 
$\| \delta \|_{H^2(\Omega)}$. 

Also, 
$$
|( \nabla\Delta\delta \mid \tilde{L}(\Delta\delta,\nabla\delta,\delta) )_{L^2(\Omega)}| 
\le C \| \nabla\Delta\delta \|_{L^2(\Omega)} 
\| \Delta\delta \|_{L^2(\Omega)} 
\| \nabla\delta \|_{L^\infty(\Omega)} \| \delta \|_{L^\infty(\Omega)}.
$$
>From the estimate 
$$
\| \nabla\delta \|_{L^\infty(\Omega)} 
\les \| \nabla\delta \|_{L^2(\Omega)} 
+ \| \nabla\Delta\delta \|_{L^2(\Omega)},
$$
we get 
$$
|( \nabla\Delta\delta \mid 
\tilde{L}(\Delta\delta,\nabla\delta,\delta) )_{L^2(\Omega)}| 
\le C ( \| \nabla\Delta\delta \|_{L^2(\Omega)} 
\| \delta \|_{H^2(\Omega)}^3 
+ \| \nabla\Delta\delta \|_{L^2(\Omega)}^2 
\| \delta \|_{H^2(\Omega)}^2 ).
$$
This leads to
$$
( \Delta^2\delta \mid 
L(\nabla\delta,\nabla\delta,\delta) )_{L^2(\Omega)} 
\le C ( \| \nabla\Delta\delta \|_{L^2(\Omega)} 
\| \delta \|_{H^2(\Omega)}^3
+ \| \nabla\Delta\delta \|_{L^2(\Omega)}^2 
\| \delta \|_{H^2(\Omega)}^2 ).
$$ 
In the same way, we have
$$
( \Delta^2\delta \mid ( \nabla m_{\rm eq}(t) \cdot 
\nabla\delta ) \delta ) 
\le C \| \nabla\Delta\delta \|_{L^2(\Omega)} 
\| \delta \|_{H^2(\Omega)}^2.
$$
\paragraph{\emph{The $\delta \wedge \Delta\delta$ term.}}  
Again,
$$( \Delta^2\delta \mid L(\delta,\Delta\delta) )_{L^2(\Omega)} 
= -( \nabla\Delta\delta \mid 
L(\nabla\delta,\Delta\delta) )_{L^2(\Omega)} 
-( \nabla\Delta\delta \mid 
L(\delta,\nabla\Delta\delta) )_{L^2(\Omega)},
$$
and as above, we get 
$$
( \Delta^2\delta \mid L(\delta,\Delta\delta) )_{L^2(\Omega)} 
\le C ( \| \nabla\Delta\delta \|_{L^2(\Omega)} 
\| \delta \|_{H^2(\Omega)}^2
+ \| \nabla\Delta\delta \|_{L^2(\Omega)}^2 
\| \delta \|_{H^2(\Omega)} ). 
$$

Finally, there is $C>0$, and for all $\eta>0$, there is $C_\eta>0$ such that 
\begin{equation*}
\begin{split}
( \delta^\eps_k \mid \mathcal{R}(m_{\rm eq})(\delta^\eps_k) )_{H^2(\Omega)} 
\le 
& \Big( \eta + C \| \delta^\eps_k \|_{H^2(\Omega)} + C \| \delta^\eps_k \|_{H^2(\Omega)}^2 \Big) 
\| \nabla\Delta\delta^\eps_k \|_{L^2(\Omega)}^2 \\
& + C_\eta \Big( \| \delta^\eps_k \|_{H^2(\Omega)}^3 + \| \delta^\eps_k \|_{H^2(\Omega)}^4 
+ \| \delta^\eps_k \|_{H^2(\Omega)}^6 \Big) .
\end{split}
\end{equation*}
Let $k$ go to infinity, so that the above estimate applies to $\delta^\eps$ instead of 
$\delta^\eps_k$, up to the local existence time $t^\eps$ obtained \emph{via} the convergence 
of the Galerkine scheme. Coming back to \eqref{eq:evolm-meq}, still with 
$\delta^\eps = m^\eps-m_{\rm eq}$, we get, using \eqref{eq:hypdissip}: 
there is $C>0$, and for all $\eta>0$, there is $C_\eta>0$ such that 
\begin{equation} \label{eq:estimdeltaepsH2}
\begin{split}
\frac{\eps}{2} \frac{\rm d}{\rm dt} 
& \left( \| \delta^\eps \|_{H^2(\Omega)}^2 \right) 
+ \Big( \alpha - \eta - C ( 1 + \| \delta^\eps \|_{H^2(\Omega)} ) 
\| \delta^\eps \|_{H^2(\Omega)} \Big) \| \nabla\delta^\eps \|_{H^2(\Omega)}^2 \\
\le 
& \Big( C_\eta ( 1 + \| \delta^\eps \|_{H^2(\Omega)} )^3 \| \delta^\eps \|_{H^2(\Omega)} 
- C_{\rm lin} \Big) \| \delta^\eps \|_{H^2(\Omega)}^2 
+ \eps^2 \| \d_t m_{\rm eq} \|_{H^2(\Omega)}^2 .
\end{split}
\end{equation}
As in the proof of Lemma~\ref{lem:gronwall}, fix $\eta\in(0,\alpha)$, and consider the time 
$\tilde{t^\eps} \le t^\eps$ up to which, in \eqref{eq:estimdeltaepsH2}, the parenthesis in 
front of $\| \nabla\delta^\eps \|_{H^2(\Omega)}^2$ (resp. $\| \delta^\eps \|_{H^2(\Omega)}^2$)  
remains positive (resp. less than $-C_{\rm lin}/2$). 
We have, for $t\in(t_\eps,\tilde{t^\eps})$: 
$$
\frac{\eps}{2} \frac{\rm d}{\rm dt} \left( \| \delta^\eps \|_{H^2(\Omega)}^2 \right) 
\le
- \frac{C_{\rm lin}}{2}  \| \delta^\eps \|_{H^2(\Omega)}^2 
+ \eps^2 \| \d_t m_{\rm eq} \|_{H^2(\Omega)}^2 .
$$
Gronwall's lemma then implies that 
$$
\sup_{[t_\eps,\tilde{t^\eps}]} \| \delta^\eps \|_{H^2(\Omega)}^2 
\le 
\| \delta^\eps(t_\eps) \|_{H^2(\Omega)}^2 
+ 2 \eps T \sup_{[0,T]} \| \d_t m_{\rm eq} \|_{H^2(\Omega)}^2,
$$
so that, for $\eps$ small enough, we get $\tilde{t^\eps} \ge T$, 
and $\sup_{[t_\eps,T]} \| \delta^\eps \|_{H^2(\Omega)} \Tend{\eps}{0} 0$. 
This finishes the proof of Theorem~\ref{th:asympt}. 
\hfill$\square$

\section{Proof of Proposition~\ref{prop:globalexist} 
and Corollary~\ref{cor:asymptbis}}
\label{sec:proofprops}

\paragraph{Proof of Proposition~\ref{prop:globalexist}.} 
For any $T>0$ and $n \in C([0,T],H^2(\Omega))$,  
it is equivalent for $n$ to be solution to \eqref{eq:n0bis} 
or to 
$$
(\d_t - \alpha \Delta) (n-m_{\rm eq}) = 
\Big( \mathcal L(0,m_{\rm eq}) 
+ \mathcal R(0,m_{\rm eq}) \Big) (n-m_{\rm eq}),
$$
with the same initial and boundary conditions. The operators 
$\mathcal L(0,m_{\rm eq})$ and $\mathcal R(0,m_{\rm eq})$ 
from \eqref{eq:defLeq} and \eqref{eq:defR} 
do not depend on time, now. 
Arguing as in Section~\ref{sec:secondstep}, we get an estimate 
analogue to \eqref{eq:estimdeltaepsH2}, 
\begin{equation} \label{eq:estimn-meqH2} 
\begin{split}
\frac{1}{2} \frac{\rm d}{\rm dt} 
& \left( \| n-m_{\rm eq} \|_{H^2(\Omega)}^2 \right) \\
+ 
& \Big( \alpha - \eta - C ( 1+\| n-m_{\rm eq} \|_{H^2(\Omega)} ) 
\| n-m_{\rm eq} \|_{H^2(\Omega)} \Big) 
\| \nabla (n-m_{\rm eq}) \|_{H^2(\Omega)}^2 \\
\le 
& \Big( C_\eta ( 1 + \| n-m_{\rm eq} \|_{H^2(\Omega)} )^3 \| n-m_{\rm eq} \|_{H^2(\Omega)} 
- C_{\rm lin} \Big) \| n-m_{\rm eq} \|_{H^2(\Omega)}^2 .
\end{split}
\end{equation}
Once $\eta \in (0,\alpha/2)$ is chosen, take $\eta_0>0$ 
such that, when $\| m_0 - m_{\rm eq} \|_{H^2(\Omega)} 
\le \eta_0$, the parentheses in front of 
$\| \nabla (n-m_{\rm eq}) \|_{H^2(\Omega)}^2$ 
and in front of $\| n-m_{\rm eq} \|_{H^2(\Omega)}^2$ 
are positive and negative at $t=0$, respectively. 
The bootstrap argument then shows that 
$n \in C([0,\infty),H^2(\Omega))$, and 
that $n(t)$ converges in $H^2(\Omega,S^2)$, 
as $t$ goes to $\infty$, towards $m_{\rm eq}(t_0)$:
\begin{equation} \label{eq:n-meqexpdecay}
\| n(t)-m_{\rm eq} \|_{H^2(\Omega)} \leq \eta_0 e^{-Ct},
\end{equation}
for some $C \in (0,C_{\rm lin})$ depending on $\eta_0$. 
Coming back to \eqref{eq:estimn-meqH2}, we see also that  
$\nabla(n-m_{\rm eq}) \in L^2((0,\infty),H^2(\Omega))$. 
\hfill$\square$

\paragraph{Proof of Corollary~\ref{cor:asymptbis}.} 
When $m_{\rm eq}(0)$ is constant over $\Omega$, 
Proposition~\ref{prop:globalexist} ensures there exists some 
$\eta_0>0$ such that for all $m_0 \in H^2_N(\Omega,S^2)$ 
satisfying 
$$
\| m_0 - m_{\rm eq} \|_{H^2(\Omega)} \le \eta_0 ,
$$
Asumption (ii) in Theorem~\ref{th:asympt} holds true. 
Furthermore, estimations \eqref{eq:estimn-meqH2} and 
\eqref{eq:n-meqexpdecay} show that the corresponding 
function $n_0$ has norms in $L^\infty((0,\infty),H^2(\Omega))$ 
and in $L^2((0,\infty)\times\Omega))$ controlled in terms 
of $m_{\rm eq}(0)$ and $\eta_0$ only. Thus, $\eps_0$ in 
the proof of Theorem~\ref{th:asympt} may also be chosen 
depending on $m_{\rm eq}(0)$ and $\eta_0$ only, uniformly 
with respect to $m_0$. 
\hfill$\square$

\section{Appendix}
\label{sec:appendix}

\subsection{About the dissipation property~\texorpdfstring{\eqref{eq:hypdissip}}{ref}: 
proof of Lemmas~\texorpdfstring{\ref{lem:dissip}}{ref} and \texorpdfstring{\ref{lem:nondissip}}{ref}} 
\label{sec:proofHypdissip}

Let $\delta \in C([0,T],H^\infty(\Omega))$ be such that  
$|m_{\rm eq}+\delta| \equiv 1$ and  
$\dn \delta_{|_{\d\Omega}} = \dn \Delta \delta_{|_{\d\Omega}} = 0$. 
Then, 
\begin{equation} \label{Lmeqpm}
{\mathcal L}(m_{\rm eq}^\pm) \, \delta = 
(\lambda \mp d) \delta \wedge u \, 
\pm \, u \wedge (\Delta\delta + h_{\rm d}(\delta)) \, 
+ \, \alpha (d \mp \lambda) u \wedge (\delta \wedge u) \,  
- \, \alpha u \wedge (u \wedge h_{\rm d}(\delta)) .
\end{equation}
\subsubsection{\texorpdfstring{$L^2$}{L2} estimates} Take the $L^2(\Omega)$ scalar 
product of \eqref{Lmeqpm} with $\delta$. This yields 
\begin{equation} \label{eq:estimL2Lmeq}
\begin{split}
\big( {\mathcal L}(m_{\rm eq}^\pm) \, \delta 
& \mid \delta \big)_{L^2(\Omega)} = 
\pm \int_\Omega \delta \cdot (u \wedge \Delta\delta) 
\pm \int_\Omega \delta \cdot (u \wedge h_{\rm d}(\delta)) \\ 
& + \alpha (d \mp \lambda) \int_\Omega |\delta \wedge u|^2 
- \alpha \int_\Omega (u \cdot h_{\rm d}(\delta)) (u \cdot \delta) 
+ \alpha \int_\Omega \delta \cdot h_{\rm d}(\delta).
\end{split}
\end{equation}
First consider the case of $m_{\rm eq}^+$.
Denoting $n$ the exterior normal vector to $\Omega$, the first term 
in the right-hand side of \eqref{eq:estimL2Lmeq} is equal to 
\begin{equation} \label{eq:stokes}
\sum_{i=1}^3 \int_\Omega \delta \cdot 
\partial_i(u \wedge \partial_i\delta) 
= \sum_{i=1}^3 \int_{\partial\Omega} \delta \cdot 
(u \wedge \partial_i\delta) n_i
= \int_{\partial\Omega} \delta \cdot 
(u \wedge \partial_n\delta)
= 0.
\end{equation}
Since $h_{\rm d}$ is continuous on $L^2$ with norm 1, the second 
term is bounded from above by $\| \delta \|_{L^2(\Omega)}^2$. 
Similarly, due to the non-positivity of $h_{\rm d}$, the last term 
is non-positive. In the two other terms, we inject the identities  
\begin{equation} \label{eq:constraintdelta}
|\delta|^2 = - 2 u\cdot\delta \quad \text{and} \quad 
|\delta \wedge u|^2 = |\delta|^2 - \frac{1}{4} |\delta|^4,
\end{equation}
which stem from the equality $|u+\delta| \equiv 1$. 
This leads to 
\begin{equation} \label{eq:estimL2Lmeq+}
\begin{split}
\left( {\mathcal L}(m_{\rm eq}^+) \, \delta 
\mid \delta \right)_{L^2(\Omega)} 
& \le \| \delta \|_{L^2(\Omega)}^2 
+ \alpha (d - \lambda) 
\int_\Omega \big( |\delta|^2 - |\delta|^4/4 \big) \\
& \hspace{3cm} + \frac{\alpha}{2} 
\int_\Omega (u \cdot h_{\rm d}(\delta)) |\delta|^2
\\
& = (1 + \alpha(d-\lambda)) \| \delta \|_{L^2(\Omega)}^2 
+ \mathcal{O} (\| \delta \|_{L^2(\Omega)}^3). 
\end{split}
\end{equation}
In the case of $m_{\rm eq}^-$, we obtain in the same way
\begin{equation} \label{eq:estimL2Lmeq-}
\left( {\mathcal L}(m_{\rm eq}^-) \, \delta 
\mid \delta \right)_{L^2(\Omega)} \ge 
\left( \alpha (\lambda + d) - c \right) \| \delta \|_{L^2(\Omega)}^2 
+ \mathcal{O} (\| \delta \|_{L^2(\Omega)}^3),
\end{equation}
for some constant $c$ depending on $\Omega$ and $\alpha$ only.

\subsubsection{\texorpdfstring{$H^2$}{H2} estimates} 
Take the $L^2(\Omega)$ scalar product of the Laplacian of each term 
in \eqref{Lmeqpm} with $\Delta\delta$. This yields 
\begin{equation} \label{eq:estimH2Lmeq}
\begin{split}
\big( \Delta {\mathcal L}(m_{\rm eq}^\pm) \, \delta 
\mid \Delta \delta \big)_{L^2(\Omega)} = 
& (\lambda \mp d) \int_\Omega \Delta\delta \cdot 
\Delta(\delta \wedge u) 
\pm \int_\Omega \Delta\delta \cdot \Delta(u \wedge \Delta\delta) \\
& \pm \int_\Omega \delta \cdot \Delta(u \wedge h_{\rm d}(\delta))  
+ \alpha (d \mp \lambda) \int_\Omega |\Delta\delta \wedge u|^2 \\ 
& - \alpha \int_\Omega \Delta\delta \cdot 
\Delta (u \wedge (u \wedge h_{\rm d}(\delta))) .
\end{split}
\end{equation} 
Since $\Delta(\delta \wedge u) = (\Delta\delta) \wedge u = 0$, 
the first term on the right-hand side vanishes. So does the second 
one, by the same argument as in \eqref{eq:stokes}. The equality 
$|u+\delta| \equiv 1$ implies 
$$
|\Delta\delta \wedge u|^2 = |\Delta\delta|^2 -
\left( |\nabla\delta|^2 - (\delta\cdot\Delta\delta)^2 \right)^2,
$$
so that \eqref{eq:estimH2Lmeq} gives, for $m_{\rm eq}^+$: 
\begin{equation} \label{eq:estimL2DeltaLmeq+}
\big( \Delta {\mathcal L}(m_{\rm eq}^+) \, \delta 
\mid \Delta \delta \big)_{L^2(\Omega)} 
\le - \alpha (\lambda-d) \|\Delta\delta\|_{L^2(\Omega)}^2 
+ c \|\delta\|_{H^2(\Omega)}^2 
+ \mathcal{O} \left( \|\delta\|_{H^2(\Omega)}^3 \right),
\end{equation}
for some constant $c$ depending on $\Omega$ and $\alpha$ only. 
Together with \eqref{eq:estimL2Lmeq+}, we get finally 
\begin{equation} \label{eq:estimH2Lmeq+}
\big( {\mathcal L}(m_{\rm eq}^+) \, \delta 
\mid \delta \big)_{H^2(\Omega)} 
\le - \left( \alpha (\lambda-d)-c \right) \|\delta\|_{H^2(\Omega)}^2 
+ \mathcal{O} \left( \|\delta\|_{H^2(\Omega)}^3 \right),
\end{equation}
which concludes the proof of Lemma~\ref{lem:dissip}. 
\hfill$\square$

In the case of $m_{\rm eq}^-$, we have 
\begin{equation} \label{eq:estimL2DeltaLmeq-}
\begin{split}
\big( \Delta {\mathcal L}(m_{\rm eq}^+) \, \delta 
\mid \Delta \delta \big)_{L^2(\Omega)} = 
& - \int_\Omega \Delta\delta \cdot \Delta(u \wedge h_{\rm d}(\delta)) 
+ \alpha (\lambda+d) \|\Delta\delta\|_{L^2(\Omega)}^2 \\
& - \alpha \int_\Omega \Delta\delta \cdot 
\Delta(u \wedge (u \wedge h_{\rm d}(\delta))) 
+ \mathcal{O} \left( \|\delta\|_{H^2(\Omega)}^3 \right),
\end{split}
\end{equation}
which, together with \eqref{eq:estimL2Lmeq-}, leads to 
Lemma~\ref{lem:nondissip}. 
\hfill$\square$

\subsection{Proof of the commutator lemma~\ref{lem:commut}} 
\label{sec:proofLemmacommut}

Writting  
$$
[P_k,\mathcal{F}(0,\cdot)](n) = (P_k-1)\mathcal{F}(0,n) 
+ \mathcal{F}(0,n) - \mathcal{F}(0,P_kn) ,
$$
the result follows from the convergence of $P_k$ 
towards $1$ pointwise as an operator on $H^1(\Omega)$ 
(which rules out the term $(P_k-1)\mathcal{F}(0,n)$) 
as well as on $H^2_N(\Omega)$, combined 
(to deal with $\mathcal{F}(0,n) - \mathcal{F}(0,P_kn)$) 
with the continuity of $\mathcal{F}(0,\cdot)$ 
from $C([0,T],H^2(\Omega)) \cap L^2((0,T),H^3)$ 
to $L^2((0,T),H^1)$. 

The latter is a consequence of the continuity properties of $h_{\rm d}$ and 
of Sobolev's embeddings, implying that $H^2(\Omega)$ is an algebra (so that 
all applications $n \mapsto n \wedge h_{\rm d}(n)$, $n \mapsto n \wedge 
(n \wedge h_{\rm d}(n))$, $n \mapsto n \wedge h_{\rm ext}(0)$, $n \mapsto n 
\wedge (n \wedge h_{\rm ext}(0))$ are continuous on $L^\infty((0,T),H^2(\Omega))$), 
and that the product operation maps $H^2 \times H^1$ to $H^1$, so that 
$n \mapsto n \wedge \Delta n$ and $n \mapsto |\nabla n|^2n$ are continuous 
from $L^\infty((0,T),H^2(\Omega)) \cap L^2((0,T),H^3)$ to $L^2((0,T),H^1)$. 
\hfill$\square$

\subsection{Proof of Gronwall's lemma~\ref{lem:gronwall}} 
\label{sec:proofLemmagronwall}

First, consider $k\in\N^\star$ and $\eps>0$ fixed. 
Set $\phi^\eps(t) = \| \varphi_k^\eps \|_{H^2(\Omega)}^2$, 
$r(t) = \tilde{r}_k^\eps(t)$, 
$N_0(t) = \| \nabla \Delta P_k n_0(t) \|_{L^2(\Omega)}^2$ 
and $N^\eps(t) = \| \nabla \Delta n_k^\eps(t) \|_{L^2(\Omega)}^2$, 
so that
$$
N^\eps(t) = N_0(t/\eps) \quad \text{and} \quad 
N_0 \in L^1(0,\infty).
$$
With $C$ from \eqref{eq:estimphikepsH2}, choose $\eta \in (0,\alpha/(2C))$. Hence, there exists 
$\kappa_\eta \in (0,1)$ such that 
$$
\forall \varphi \in [0,\kappa_\eta], \quad 
C (\eta + \varphi (1+\varphi)) < \alpha /2 .
$$
Set $K=8C_\eta$ (also from \eqref{eq:estimphikepsH2}), 
$c\in(0,1/K)$ and $t_\eps = c\eps\ln(1/\eps)$. Then, with  
$$
t_k^\eps = \sup\{ t \in [0,t_\eps] \mid \forall t'\in[0,t], \, 
\phi^\eps(t') \le \kappa_\eta \}
$$
($t_k^\eps>0$ since $\phi^\eps(0)=0$), we have:
$$
\forall t\in[0,t_k^\eps], \quad \eps {\phi^\eps}'(t) \le 
K ( (1+N^\eps(t))\phi^\eps(t) + t + r(t) ) .
$$
>From this, we deduce: 
\begin{equation*}
\begin{split}
\forall t\in[0,t_k^\eps] , \quad \phi^\eps(t)  
& \le \int_0^t \frac{K}{\eps} \Big( t' + r(t') \Big)  
\exp\left( \frac{K}{\eps} \int_{t'}^t (1+N^\eps(t''))dt'' \right) 
dt' \\ 
& \le \left( \int_0^t \frac{K}{\eps} \Big( t' + r(t') \Big)  
\exp\left( K\frac{t-t'}{\eps} \right) dt' \right) 
e^{K \| N_0 \|_{L^1(0,\infty)}} \\ 
& \le \left( \frac{\eps}{K} + \frac{K}{\eps} 
\| r \|_{L^1(0,T_0)} \right) 
e^{Kt/\eps} e^{K \| N_0 \|_{L^1(0,\infty)}} ,
\end{split}
\end{equation*}
with $T_0 = c \eps_0 \ln(1/\eps_0)$ (and $\eps_0$ is chosen below). 
Now, since $c\in(0,1/K)$ and $t_\eps = c\eps\ln(1/\eps)$, 
$$
\forall t\in[0,\min(t_k^\eps,t_\eps)] , \quad 
\phi^\eps(t) \le \left( \frac{\eps^{1-cK}}{K} + K \| r \|_{L^1(0,T_0)} \eps^{-1-cK} \right) e^{K \| N_0 \|_{L^1(0,\infty)}} ,
$$
which is less or equal to $\kappa_\eta$ as soon as $\eps$ belongs 
to $(0,\eps_0]$, for 
$$
\eps_0 = \left( \frac{1}{2}\kappa_\eta K 
e^{-K \| N_0 \|_{L^1(0,\infty)}} \right)^{1/(1-cK)} ,
$$ 
and $k$ greater than $K(\eps)$ such that 
$$
\forall k \ge K(\eps), \quad \| \tilde{r}_k^\eps \|_{L^1(0,T_0)} 
\le \frac{1}{2}\kappa_\eta \frac{\eps^{1+cK}}{K} 
e^{-K \| N_0 \|_{L^1(0,\infty)}}
$$
(which is possible by Lemma~\ref{lem:commut}). 
For this choice of $\eps$ and $k$, we thus have 
$t_k^\eps \ge t_\eps$, and the result follows. \hfill$\square$

\bibliographystyle{plain}
\bibliography{bibliohysteresis}

\end{document}